\documentclass{amsart}
\usepackage[all,cmtip]{xy}
\usepackage{amssymb}
\usepackage{mathtools}
\usepackage{amsthm}
\usepackage{tikz-cd}
\usepackage{hyperref}

\usepackage{cleveref}

\newtheorem{theorem}{Theorem}[section]

\newtheorem{lemma}[theorem]{Lemma}
\newtheorem{proposition}{Proposition}[section]
\newtheorem{remark}{Remark}[section]
\newtheorem{definition}{Definition}[section]
\newtheorem{example}{Example}[section]

\def\ddefloop#1{\ifx\ddefloop#1\else\ddef{#1}\expandafter\ddefloop\fi}

\def\ddef#1{\expandafter\def\csname bb#1\endcsname{\ensuremath{\mathbb{#1}}}}
\ddefloop ABCDEFGHIJKLMNOPQRSTUVWXYZ\ddefloop

\ddefloop ABCDEFGHIJKLMNOPQRSTUVWXYZ\ddefloop

\def\ddef#1{\expandafter\def\csname ff#1\endcsname{\ensuremath{\mathfrak{#1}}}}
\ddefloop ABCDEFGHIJKLMNOPQRSTUVWXYZ\ddefloop
\def\ddef#1{\expandafter\def\csname cc#1\endcsname{\ensuremath{\mathcal{#1}}}}
\ddefloop ABCDEFGHIJKLMNOPQRSTUVWXYZ\ddefloop

\newcommand{\cart}{\ar@{}[dr]|{\Box}}

\newcommand{\cs}{{h^1}/{h^0}}
\newcommand{\si}{\sigma}

\newcommand{\bC}{\mathbb{C}}
\newcommand{\bE}{\mathbb{E}}
\newcommand{\bZ}{\mathbb{Z}}
\newcommand{\bA}{\mathbb{A}}
\newcommand{\bT}{\mathbb{T}}

\newcommand{\bL}{\mathbb{L}}
\newcommand{\bP}{\mathbb{P}}
\newcommand{\Ad}{\mathrm{Ad}}
\newcommand{\End}{\mathrm{End}}
\newcommand{\Ext}{\mathrm{Ext}}
\newcommand{\Hom}{\mathrm{Hom}} 
\newcommand{\pr}{\mathrm{pr}}
\newcommand{\Id}{\mathrm{Id}}
\newcommand{\Crit}{\mathrm{Crit}}

\newcommand{\vir}{ {\mathrm{vir}} }
\newcommand{\Ob}{ \mathrm{Ob}}

\newcommand{\ev}{\mathrm{ev}} 
\newcommand{\Spec}{\mathrm{Spec}}
\newcommand{\Sym}{\mathrm{Sym}}
\newcommand{\Vb}{\mathrm{Vb}}
\newcommand{\fix}{\mathrm{fix}}
\newcommand{\mov}{\mathrm{mov}}
\newcommand{\tw}{\mathrm{tw}}
\newcommand{\age}{\mathrm{age}}
\renewcommand{\deg}{\mathrm{deg}}
\newcommand{\Aut}{\mathrm{Aut}}

\newcommand{\fgl}{\mathfrak{gl}}
\newcommand{\fB}{\mathfrak{B}}
\newcommand{\fC}{\mathfrak{C}}
\newcommand{\fE}{\mathfrak{E}}
\newcommand{\fM}{\mathfrak{M}}
\newcommand{\fS}{\mathfrak{S}}
\newcommand{\fP}{\mathfrak{P}}

\newcommand{\fV}{\mathfrak{V}}
\newcommand{\fX}{\mathfrak{X}}
\newcommand{\fZ}{\mathfrak{Z}}
\newcommand{\Bun}{\mathfrak{B}}
\newcommand{\fN}{\mathfrak{N}}

\newcommand{\cE}{\mathcal{E}}
\newcommand{\cO}{\mathcal{O}}
\newcommand{\cP}{\mathcal{P}}
\newcommand{\cZ}{\mathcal{Z}}
\newcommand{\cL}{\mathcal{L}}
\renewcommand{\ccX}{\mathcal{X}}

\newcommand{\tW}{\widetilde{W}}
\newcommand{\tE}{\widetilde{E}}

\newcommand{\lra}{\longrightarrow}
\newcommand{\fixd}{\mathrm{fix}}
\newcommand{\movn}{\mathrm{mov}}
\title{Moduli of stable maps with fields}
\author{Renata Picciotto}
\makeindex
\begin{document}
\maketitle
\begin{abstract}
Given  a triple $(X,E,s)$ of a smooth projective variety, a rank $r$ vector bundle and a regular section, we construct a moduli of stable maps to $X$ with fields together with a cosection localized virtual class. We show the class coincides up to a sign with the virtual fundamental class on the moduli space of stable maps to $Z=Z(s)\subset X$. We show that this gives a generalization of the Quantum Lefschetz hyperplane principle. We further generalize this result by considering $(\ccX,\ccE,s)$ where $\ccX$ is a smooth Deligne--Mumford stack with projective coarse moduli space. In this setting, we can construct a moduli space of twisted stable maps to $\ccX$ with fields, with components of constant virtual dimension indexed by $n$-tuples of components of the inertia stack of $\ccX$. This generalizes similar comparison results of Chang--Li, Kim--Oh and Chang--Li and presents a different approach from Chen--Janda--Webb. 

\end{abstract}

\tableofcontents

\section{Introduction}

The Quantum Lefschetz (QL) principle, as conjectured by Cox--Katz--Lee \cite{conj} and proved by Kim--Kresch-- Pantev in \cite{KKP}, is a statement about the functoriality of virtual fundamental classes. Let $Z=\{s=0\}\subset X$ be smooth projective varieties with $s$ a regular section of a convex vector bundle $E$ on $X$, that is one for which $H^1(C,f^*E)=0$ for any $f:C\to X$ for any genus 0 stable map, then the $g=0$ Gromov--Witten invariants of $Z$ can be obtained from those of $X$ by computing the Euler class of a vector bundle derived from $E$. More precisely:
\[
\sum_{i_*\beta'=\beta}\iota_{\beta' *}[\overline{\ccM}_{0,n}(Z,\beta')]^\vir=c_{\mathrm{top}}(\ccV_{\beta,n})\cap[\overline{\ccM}_{0,n}(X,\beta)]^{\vir}.
\]
The bundle $\ccV_{\beta,n}$ on $\overline{\ccM}_{0,n}(X,d)$ is obtained by pulling back $E$ through the universal evaluation map $ev_{n+1}: \overline{\ccM}_{0,n+1}(X,\beta)\to X$ and pushing it forward by the forgetful morphism $\pi_{n+1}:\overline{\ccM}_{0,n+1}(X,\beta)\to \overline{\ccM}_{0,n}(X,\beta)$.

Some similar results in genus 1 have been formulated for reduced Gromov--Witten invariants in \cite{li2005genus} and \cite{chang2015algebraic}, but a straight-forward generalization the above ``hyperplane property" fails for higher genus Gromov--Witten invariants because in general $\ccV_{d,n}$  fails to be a vector bundle.  

The idea for a different approach originated in the physics literature due to the work of Guffin--Sharpe \cite{GS}.  At the level of target spaces, the section $s$ of $E$ gives a cosection $W:E^{\vee}\to\bC$ with critical locus $Z$. They consider a theory of $A$-twisted Landau--Ginzburg model defined by $(X,E,s)$ and compare it to the usual Gromov--Witten theory of $Z$. Mathematically, we would like to consider some non-proper moduli space $\fX^E$ playing the role of $\Vb(E^{\vee})$ and a virtual fundamental class that is the analog of the class of the degeneracy locus $\bbD(\sigma)\in A_*(Z)$. The difficult part is to compare $[\fX^E]^{\vir}_{\sigma}$ to the usual fundamental class of $\fZ$.
This was undertaken by Chang--Li for the case of the quintic threefold $Q$. In \cite{CL} they propose a new moduli space 
\[
\overline{\ccM}_g(\bP^4,d)^p = \{ [u,C,p] | [u,C]\in\overline{\ccM}_g(\bP^4,d), p\in\Gamma(C,u^*\cO_{\bP^4}(-5)\otimes\omega_C)\}/\sim
\]
of stable maps to $\bP^4$ with $p$-fields. This moduli space is a cone over the usual stable maps to $\bP^4$ and, although not proper, contains the proper moduli space $\overline{\ccM}_g(Q,d)$ as the degeneracy locus of a ``superpotential'' determined by the polynomial $s$. Using the techniques developed by Kiem--Li in \cite{KL}, the authors construct a cosection localized virtual class $[\overline{\ccM}_g(\bP^4,d)^p]^\vir_\sigma$ supported on $\overline{\ccM}_g(Q,d)$. By degeneration to the normal cone, they prove the degree of this class agrees with that of the usual virtual class $[\overline{\ccM}_g(Q,d)]^\vir$ up to a sign determined by $g$ and $d$. 

This result has since been variously generalized. In \cite{chang2020invariants} Chang and M.-L. Li consider a complete intersection $Z$ in a projective space $\bbP^m$ defined by some polynomials $(f_1,\dots , f_r)$, sections of line bundles $\ccO(d_1)\dots, \ccO(d_r)$. They consider the moduli of sections $(u_1,\dots ,u_{m+1})$ where $u_i\in H^0(C,L)$ and $L$ is a line bundle on a nodal curve $C$. The family of stability conditions that can be imposed to make sure this stack is Deligne-Mumford is parametrized by $\epsilon\in\bbQ_{>0}\cup \{0+\}$, where $1/\epsilon$ controls the allowed lengths of common zeros of $(u_1,\dots , u_{m+1})$. Each stability condition produces a (not necessarly different) moduli space of $\epsilon$-stable quasi-maps to $\bbP^m$, $\overline{\ccM}^{\epsilon}_{g,n}(\bbP^m,d)$. Requiring that $(u_1,\dots , u_{m+1})$ satisfy the polynomials $(f_1,\dots , f_r)$ gives the moduli spaces $\overline{\ccM}^{\epsilon}_{g,n}(Z,d)$ of  $\epsilon$-stable quasi-maps to $Z$. Adding a section $p\in H^0(C,\bigoplus_{i=1}^rL^{\otimes -d_i}\otimes\omega_C)$ to quasi-maps to $\bbP^m$ they obtain the stack of quasi-maps to $\bbP^m$ with $p$-fields, $\overline{\ccM}^{\epsilon}_{g,n}(\bbP^m,d)^p$. This has a cosection localized virtual class $[\overline{\ccM}^{\epsilon}_{g,n}(\bbP^m,d)^p]_\sigma^\vir$, where the cosection is determined by the gradients of the polynomials $(f_1,\dots ,f_r)$. They prove that for any choice of stability conditions $\epsilon$, the virtual classes $[\overline{\ccM}^{\epsilon}_{g,n}(\bbP^m,d)^p]_\sigma^\vir$ and $[\overline{\ccM}^{\epsilon}_{g,n}(Z,d)]^\vir$ agree up to a sign determined by $d_1,\dots, d_r, g$ and $d$. In particular, for the Gromov--Witten stability conditions corresponding to $\epsilon=\infty$, this generalizes the result of \cite{CL} and rephrases it in terms of virtual classes, not just degrees, by a proving a more powerful result about virtual pullback for cosection localized classes. 

Quasi-maps can be defined in a more general settings, for example with target $X=V_1/\!\!/G$, the GIT quotient of a vector space $V_1$ by a reductive algebraic group $G$ \cite{ciocan2014stable}. In this setting a quasi-map from a prestable curve $C$ with a fixed $G$-torsor $P$ is a $G$-equivariant morphism $P\to V_1$. As in the $\bbP^m$ case above, there is a family of stability conditions parametrized by $\epsilon$ which for large values recovers the Gromov--Witten stability conditions. In \cite{kim2018localized}, Kim--Oh consider a smooth codimension $r$ subvariety $Z$ of such a $X$ given by the vanishing of a section $s$ of a vector bundle $E=[V_1^{\mathrm{ss}}(\theta)\times V_2^{\vee}/G]$. They prove via localized Chern complexes that the cosection localized class of the moduli space of $\epsilon$-stable quasi-maps to $X$, denoted $[\overline{\ccM}^\epsilon(X,\beta)^p]^\vir_\sigma$, agrees up to a sign with $\sum_{i_*\beta'=\beta}[\overline{\ccM}^\epsilon_{g,n}(Z,\beta')]^\vir$.

A further generalization of \cite{CL} has appeared in the work of Chen--Janda--Webb \cite{chen2019virtual}, who consider the moduli space of maps to $Z$ which is the smooth vanishing locus of a regular section $s$ of a vector bundle $E$ on a smooth projective DM stack $X$. In contexts where the theory of quasi-maps is well defined, such as when $X=W/\!\!/G$ for an l.c.i. affine variety $W$ and reductive algebraic group $G$, they prove a generalization of the quasi-maps results of \cite{kim2018localized} and \cite{chang2020invariants}. In the general set-up where quasi-maps are not defined, they show that the Gromov--Witten virtual classes $[\overline{\ccM}(X,\beta)^p]^\vir_\sigma$ and $\sum_{i_*\beta'=\beta}[\overline{\ccM}_{g,n}(Z,\beta')]^\vir$ agree up to sign.  Their more general result is proven by considering the moduli spaces of stable (quasi-)maps as moduli of sections relative to the stack of prestable curves $\mathfrak{M}_{g,n}$. Our approach avoids some of the technical difficulties they encounter and gives a more explicit formula for the perfect obstruction theory and the cosection by working relative to the moduli of prestable curves with a fixed rank $r$ degree $d$ vector bundle $\mathfrak{B}_{g,n, r, d}$. This leads to considering the deformation theory of stable maps $C\to X$ which pull back $E$ on $X$ to a fixed bundle $V$ on $C$. 

Section 2 will cover some of the background of virtual fundamental classes and cosection localized virtual classes. We will recall some functoriality properties that we will use throughout.

In Sections 3--6 we will closely follow and generalize \cite{CL} and \cite{chang2020invariants} by replacing the ambient space $\bP^m$ by a smooth projective variety $X$ without the choice of an ample line bundle and a fixed projective embedding. We fix a vector bundle $E$ on $X$ and a section $s$ with smooth vanishing locus $Z$ of the expected dimension. From the data of the triple $(X,E,s)$, and having fixed genus, degree and number of marked points $(g,n,\beta)$, we construct the moduli of stable maps to $X$ with fields in $E$, denoted $\fX^E$:
\[
\fX^E=\{[(C,\underline{x}),f,p]|[(C,\underline{x}),f]\in\overline{\ccM}_{g,n}(X,\beta),p\in H^0(C,f^*E^{\vee}\otimes\omega_C)\}/\sim
\]
 The triple $(X,E,s)$ gives rise to the data of the moduli space $\fX^E$ together with a perfect obstruction theory and a superpotential, or cosection, with critical locus 
\[
\fZ:=\bigcup_{\beta'|i_*\beta'=\beta}\overline{\ccM}_{g,n}(Z,\beta').
\]
Then there is a cosection localized virtual fundamental class $[\fX^E]^{\vir}_{\sigma}$.

The technique of deformation to the normal cone applied in \cite{CL} and subsequent papers works well in this set-up and allows to pass from $\fX^E$ to the moduli space $\fN^N$ which is an abelian cone of virtual dimension 0 over $\fZ$ defined as
\[
\fN^N=\{[(C,\underline{x}),f,p,q]|[(C,\underline{x}),f]\in\fZ, p\in\Gamma(C,f^*N^\vee\otimes\omega_C), q\in\Gamma(C,f^*N)\}/\sim.
\]
The obstruction sheaf of $\fN^N$ relative to its zero-section $\fZ$ has fiber $(\Ob_{\fN^N/\fZ})_{[C,f,p,q]}=H^1(C,f^*N^\vee\otimes\omega_C)\oplus H^1(C,f^*N)$ corresponding to the obstructions to deforming given sections $p$ and $q$. The cosection of $\fX^E$ deforms to the cosection of $\fN^N$ given by Serre duality:
\begin{align*}
(\Ob_{\fN^N/\fZ})_{[C,f,p,q]}=H^1(C,f^*N^\vee\otimes\omega_C)\oplus H^1(C,f^*N)&\lra  H^1(C,\omega_C)\cong\bC\\
(\dot{p},\dot{q})&\mapsto \langle \dot{p},q\rangle+\langle\dot{q}p\rangle.
\end{align*}
This cosection is invariant under the $\bC^*$ action on the cone $\fN^N$, so allows for cosection localized virtual localization to the fixed locus which is the 0-section $\fZ$. The result is that the cosection localized class $[\fN^N]_\sigma^\vir$ agrees up to a sign with the virtual class of $\fZ$. The sign is determined by the virtual normal bundle of $\fZ$ in $\fN^N$ which is $R^\bullet\pi_{\fZ *}\cE_{\fZ}\oplus\cE_{\fZ}^{\vee}\otimes\omega_{\pi_{\fZ}}$. So the main result is that the cosection localized virtual class on $\fX^E$ agrees with the usual virtual class on the moduli space of stable maps to $\fZ$. 
\begin{theorem}
\[
[\fX^E]^{\vir}_{\sigma}=(-1)^{\int_{\beta}c_1(E)+r(1-g)}[\fZ]^{\vir}.
\]

\end{theorem}

Section 7 will follow the approach of the previous sections to further generalize the result to the case where $\ccX$ is a smooth Deligne--Mumford (DM) stack with projective coarse moduli space and $s$ is a regular section of some bundle $\ccE$ (not necessarly pulled back from the coarse space). Due to the subtleties of working with stable maps to a DM stack, which require allowing some orbifold structure on the domain curves, we have chosen to treat this case separately to avoid further encumbering the notation of the previous sections. The spaces we will consider are those of twisted stable maps with a fix trivialization of the marked gerbes. The main difference from the previous case is that we wish to keep track of the actions of the isotropy groups of the marked points on $f^*\ccE$. Doing this decomposes our moduli spaces into (possibly disconnected) components of constant virtual dimension, $\fX^{\ccE}(\lambda_1,\dots,\lambda_n)$ and $\fZ^{\tw}(\lambda_1,\dots,\lambda_n)$. The main result of this section will then be that for each of these components we have
\begin{theorem}
\[
[\fX^{\ccE}(\lambda_1,\dots,\lambda_n)]^{\vir}_{\sigma}=(-1)^{r(1-g)+d-\sum_{i=1}^n\age_{\ccX_{\lambda_i}}(\ccE)}[\fZ^{\tw}(\lambda_1,\dots,\lambda_n)]^{\vir}.
\]
\end{theorem}
The formula for the sign comes from the Riemann--Roch theorem for orbi-curves.
Although it is well-known that the QL principle can fail in higher genus, it is still somewhat surprising that it can also fail for the genus 0 Gromov--Witten theory of a hypersurface in orbifold projective space. Some examples of this failure are detailed in \cite{MR3039825}. In the last section of we will show how our construction applies to their specific examples.

\subsection{Acknowledgments}
I am grateful to my advisor Chiu-Chu Melissa Liu for suggesting this project to me and for patiently helping me through the stages of this work. I would also like to thank Johan de Jong for the very helpful conversations and Barbara Fantechi for her inspiring course on virtual fundamental classes at Berkeley University in the spring of 2018. My gratitude also goes to Henry Liu and Andrea Dotto for their helpful comments and suggestions, as well as Qile Chen, Felix Janda and Rachel Webb for the talks they gave at Columbia explaining their work as well as the conversations that followed. 
\section{Notation and background}
\subsection{Notation}
We work over $\bC$. We introduce some notation.
\begin{enumerate}
\item Let $GL_r$ denote the general linear group of $r\times r$ invertible matrices, and let $\fgl_r$ denote the Lie algebra of $GL_r$.
\item  If a group $G$ acts on a scheme $M$, we let $[M/G]$ denote the quotient stack and let $M/G$ denote the coarse
moduli space. There is a projection $[M/G]\lra M/G$, which is an isomorphism when the $G$-action on $M$ is free.
\item For a locally free sheaf $\cE$ of constant rank on $X$, we adopt the convention $\Vb(\cE)=\Spec_X(\Sym\cE^\vee)$. 
\item Given a locally free sheaf $\ccE$ over $X$ of rank $r$, we denote by $P_{\cE}$ the corresponding $GL_r$-torsor, where we fix the fundamental representation of $GL_r$. 
\item For any scheme or stack $M$, let $\Id_M: M\to M$ denote the identity map. 
\item Let $\bullet$ denote a point.
\item $\fM_{g,n}$ is the Artin stack of prestable genus $g$, $n$-pointed curves with universal family $\fC_{\fM_{g,n}}$. For a stack $\fX$ over $\fM_{g,n}$ we denote the universal curve of $\fX$ by $\fC_{\fX}$ and the projection $\fC_{\fX}\to\fX$ by $\pi_{\fX}$. If $\fX$ is a moduli of stable maps to a target, we denote the universal evaluation morphism by $ev_{\fX}$.
\item We will often denote morphism of stacks which come from forgetful functors by $\pi_{\fX/\fM}:\fX\to \fM$. If $\fM$ has a universal curve $\fC_{\fM}$, the corresponding morphism $\fC_{\fX}\to \fC_{\fM}$ will be denoted by $\overline{\pi_{\fX/\fM}}$.
\item $\ccD_{M}$ will denote the derived category of quasicoherent sheaves on $M$. $\bL_M$ will denote the cotangent complex of $M$ and $\bT_M$ its derived dual. 
\end{enumerate}
\subsection{Background}
We collect here some well-known results about virtual fundamental classes and cosection localized classes.

Let $f: M\to S$ be a morphism of DM-type between algebraic stacks and $(\phi_{M/S},\bE_{M/S})$ {a perfect obstruction theory} in the sense of \cite{BF}. Manolache \cite{VPB} constructs a virtual pull-back
\[
f^!_{\bE_{M/S}}: A_*(S)\to A_{*-n}(M).
\]
If $S$ is an Artin stack of pure dimension, we get $[M]^{\vir}=f^!_{\bE_{M/S}}[S]$ . Let $g:N\to M$ be a DM-type morphism with perfect obstruction theory $\bE_{N/M}$. Let $h=f\circ g:N\to S$ have some perfect obstruction theory $\bE_{N/S}$. We say the triple $(\bE_{N/S},\bE_{M/S},\bE_{N/M})$ is a \textit{compatible triple} of perfect obstruction theories if there is a morphism of distinguished triangles
\[
\xymatrix{
g^*\bE_{M/S}\ar[r]\ar[d]^{g^*\phi_{M/S}} & \bE_{N/S}\ar[r]\ar[d]^{\phi_{N/S}} & \bE_{N/M}\ar[d]^{\phi_{N/M}}\ar[r] & \\
g^*\bL_{M/S}\ar[r] & \bL_{N/S}\ar[r] & \bL_{N/M}\ar[r] & 
}
\]
then we have the following functoriality result 
\begin{theorem}{\cite[Theorem 4.8]{VPB}}
For $\alpha\in A_*(S)$,
 \[
h^!_{\bE_{N/S}}(\alpha)=g^!_{\bE_{N/M}}(f^!_{\bE_{M/S}}(\alpha)).
\]
\end{theorem}

We also summarize here some constructions of Chang--Li \cite[Section 2]{CL}, which we will use throughout. Let $S$ be an Artin stack with $\pi:\ccC\to S$ a flat family of connected, nodal, arithmetic genus $g$ curves. Let $\ccL$ a locally free sheaf on $\ccC$. The \textit{direct image cone} is
\[
C(\pi_*\ccL):=\Spec_{S}\Sym^{\bullet}( R^1\pi_*\ccL^{\vee}\otimes\omega_{\pi}).
\]
For any scheme $f:T\to S$ the fiber $C(\pi_*\ccL)(T)$ is the collection of $p\in H^0(C_T,f^*\ccL)$ where $C_T=T\times_{S}\ccC$ and $f^*\ccL=\ccL\times_{\cO_{\ccC}}\cO_{C_T}$. This follows by Serre duality and by the fact that in this case $R^1\pi_*$ commutes with pullbacks. 
More generally, the \textit{moduli of sections} of a representable, smooth, quasi-projective morphism $Z\to \ccC$ is the groupoid
\[
\Gamma(Z/\ccC)(T):=\{\mbox{sections of } Z_T\to C_T \}.
\]
which is an algebraic stack, representable and quasi-projective over $S$. 
For example, $C(\pi_*\ccL)$ above is the moduli of sections of $\Vb(\ccL)$.

\begin{theorem}{\cite[Proposition 2.5]{CL}} The morphism
$\Gamma(Z/\ccC)\to S$ has a relative perfect obstruction theory given by
\[
\phi_{\Gamma(Z/\ccC)/S}:\bL_{\Gamma(Z/\ccC)/S}\to \bE_{\Gamma(Z/\ccC)/S}=\left(R^{\bullet}\pi_{\Gamma(Z/\ccC)*}\mathfrak{e}^* T_{Z/\ccC} \right),
\]
where $\mathfrak{e}:\ccC_\Gamma(Z/\ccC)\to Z$ is the tautological evaluation map.The morphism $\phi_{\Gamma(Z/\ccC)/S}$ is induced by the differential of $\mathfrak{e}$.
\end{theorem}

To construct the virtual class of a degeneracy locus in a non-proper space, the quantum analog of $\Vb(E^{\vee})$, we will use the technique developed by Kiem--Li \cite{KL} and Chang--Kiem--Li \cite{CKL}. We briefly discuss the construction and summarize its main properties.

\begin{theorem}\cite{KL}\label{thm:coslocvir}
Let $M$ be a Deligne--Mumford stack with a perfect obstruction theory $\bE_{M}$ of virtual dimension $d^{\vir}_M$ and a cosection of the obstruction sheaf
\[
\sigma:h^1(\bE_M^{\vee})\to\cO_M.
\]
There is a cycle $[M]^{\vir}_{\sigma}\in A_{d^{\vir}_M}(D(\sigma))$ supported on the degeneracy locus 
$D(\sigma)$
 of the cosection. Moreover, if $[M]^\vir$ exists, then $i_*[M]^\vir_{\sigma}=[M]^{\vir}$ for $i:D(\sigma)\xhookrightarrow{} M$.
\end{theorem}
Note that the theorem is stated for an absolute perfect obstruction theory and obstruction sheaf. However, we often consider $M$ relative to some smooth Artin stack of pure dimension, say $\pi:M\to S$. Let $\phi_{M/S}:\bE_{M/S}\to\bL_{M/S}$ be a relative perfect obstruction theory. We can recover an absolute obstruction sheaf for $M$ as follows. Consider the distinguished triangle
\[
\bL_{M}\to\bL_{M/S}\to\pi^*\bL_{S}[1]\to
\]
Composing $\phi_{M/S}$ with the last morphism gives $g: \bE_{M/S}\to \pi^*\bL_{S}[1]$. Then the shifted mapping cone $\bE_M:=c(g)[-1]$ gives an absolute perfect obstruction theory for $M$ (see, for example \cite[Proposition 3]{KKP}). Then we take \[
\Ob_{M}=h^1(\bE_{M}^{\vee})=\mathrm{Coker}\{(h^1(g^{\vee})):h^0(\pi^*\bT_{S})\to\Ob_{M/S}\}.
\]
Note that the absolute perfect obstruction theory depends on the relative perfect obstruction theory. 
For the relative version of \cref{thm:coslocvir} to hold, we require a relative cosection
\[
 \sigma_{M/S}:\Ob_{M/S}\to\ccO_{M}
\]
which factors through $\Ob_{M/S}\to\Ob_{M}$. Note that the latter morphism is surjective, so the degeneracy locus is unchanged.
\begin{remark}\label{rklift}
All our cosections will arise as $\sigma_{M/S} = h^1(\sigma^{\bullet})$ from some morphism $\sigma^{\bullet}:\bE^{\vee}_{M/S}\to \bbF$, where $\bbF$ will be some perfect complex concentrated in degrees $[0,1]$ with $h^1(\bbF)=\cO_M$. To check such a cosection factors, it suffices to check that the composition $h^1(\sigma^{\bullet}\circ \phi^{\vee}_{M/S})$ is trivial.
\end{remark}

\begin{definition}
For relative perfect obstruction theories $\bE_{M/S}$ and $\bE_{M/T}$ relative to different smooth Artin stacks $S$ and $T$, we will call the obstruction theories compatible if they induce the same absolute obstruction theory. We will say that cosections of the relative obstruction sheaves are compatible if they both factor through $\Ob_M$ as the same absolute cosection. 
\end{definition}

The construction of this cosection localized cycle has two main steps. First Kiem--Li define a cosection localized Gysin map
\[
s^!_{\bE_{M},\sigma}: A_*(\underline{E}(\sigma))\to A_{*+ d^{\vir}_M}(D(\sigma))
\]
where 
\[
\underline{E}(\sigma):=\cs(\bE_{M})|_{D(\sigma)}\sqcup \mathrm{ker}(\sigma|_{M\setminus D(\sigma)}).
\]
Then they show that the instrinsic normal cone $\ccC_M$ of $M$ is supported on $\underline{E}(\sigma)$, if the cosection lifts to a cosection of the absolute obstruction bundle. They define
\[
[M]^{\vir}_{\sigma}=s^!_{\bE_{M},\sigma}([\ccC_M]).
\] 
For cycles supported on $\cs(\bE_{M})|_{D(\sigma)}$, $s^!_{\bE_{M},\sigma}$ is just given by intersecting with the zero-section. More generally, suppose for simplicity that there is a locally-free sheaf $E_1\to\Ob_{M}\to 0$. For any cycle $[B]\in Z_*(E_1(\sigma))$ represented by a closed integral substack, they show one can always construct a \textit{regularizing morphism} $\nu:\tilde{M}\to M$ such that there is a short exact sequence of locally-free sheaves
\[
0\to K\to \nu^*E_1\to \cO_{\tilde{M}}(D)\to 0
\]
where the last morphism is induced by $\sigma$ and $D$ is a Cartier divisor. In this case, they show that there is a closed integral substack $\tilde{B}\subset K$ such that $\nu_*[\tilde{B}]=k[B]$ for some integer $k$. Then
\[
s^!_{\bE_{M},\sigma}[B]:=\frac{1}{k}\nu_*([D]\cdot 0^!_{K}[\tilde{B}]).
\]

We will mainly use the following three theorems about cosection localized virtual fundamental classes. The first is a functoriality result. 
\begin{theorem}{\cite[Theorem 2.6]{CKL}}\label{theoremfun}
Let $f:M\to N$ be a virtually smooth morphism of DM stacks. Suppose that, for some smooth Artin stack $S$ of pure dimension, there is a compatible triple of perfect obstruction theories $(\bE_{N/S},\bE_{M/S},\bE_{M/N})$ and that $\sigma:\Ob_{N}\to \cO_N$ is an absolute cosection. This induces a cosection on $M$ given by 
\[
\sigma': \Ob_M\to f^*\Ob_N\to f^*\cO_N=\cO_M.
\]
Then
\[
f^!_{\bE_{M/N}}[N]^{\vir}_{\sigma}=[M]^{vir}_{\sigma'}.
\]
where $f^!_{\bE_{M/N}}$ is the virtual pullback of Manolache \cite{VPB}.
\end{theorem}

The second is a deformation-invariance result.
Although this result was established first, and is used in the proof of \cref{theoremfun}, it is useful to view it as a special case of the functoriality statement. 
Let $N$ be a DM stack with a morphism $N\to \bA^1\times S$ where $S$ is a smooth Artin stack. Let $N_t$ be the fiber of $N$ over $t\in\bA^1$:
\[
\xymatrix{
N_t\ar[r]^{\iota_t}\ar[d]& N\ar[d]\\
t\ar[r]^{i_t} & \bA^1.
}
\] 
Let $\bE_{N/S}$ and $\bE_{N/S\times\bA^1}$ be compatible relative perfect obstruction theories on $N$ in the sense that we have a morphism of distinguished triangles
\begin{equation}\label{a1inv}
\xymatrix{
\ccO_N[-1]\ar[r]& \bE^{\vee}_{N/\bA^1\times S} \ar[r] & \bE^{\vee}_{N/S}\ar[r]^{+1} & \\
\ccO_N[-1]\ar[r]\ar[u]^{\cong}& \bT_{N/\bA^1\times S} \ar[u]^{\phi^{\vee}_{N/\bA^1\times S}}\ar[r] &\bT_{N/S}\ar[u]^{\phi^{\vee}_{N/S}}\ar[r]^{+1} &
}
\end{equation}
 In particular, the absolute obstruction sheaf $\Ob_N$ induced by either is the same. Let $\sigma$ be an absolute cosection of $\Ob_N$. Then we can define a perfect obstruction theory for $N_t$ by 
\[
\bE_{N_t/S}:=\iota_t^*\bE_{N/S\times\bA^1}\to\iota_t^*\bL_{N/S\times\bA^1}\to\bL_{N_t/S}.
\]
  Let $\bE_{N_t/N}=\cO_{N_t}[1]$. The triple $(\bE_{N/S},\bE_{N_t/S},\bE_{N_t/N})$ is compatible. Let $\sigma_t$ induced by $\sigma$ as in the previous section. Then by \cref{theoremfun} we have
\begin{theorem}{\cite[Theorem 5.2]{KL}}\label{deformationthm}
\[
[N_t]^{\vir}_{\sigma_t}=i_t^![N]^{\vir}_{\sigma}
\]
where $i_t^!$ is a Gysin map of the divisor $t\in \bA^1$.
\end{theorem}
Note that the two different obstruction theories for $N$ are both needed for this construction: the one relative to $S\times\bA^1$ to induce a perfect obstruction theory on the fiber and the one relative to $S$ to obtain the compatible triple needed for the functoriality statement.

Finally, there is a cosection localized version of the usual torus localization result by Graber--Pandharipande \cite{Loc}.
Let $T=\bC^*$ and $N$ be a DM stack with a $T$-action and $T$-equivariant perfect obstruction theory $\phi_N:\bE_{N}\to\bL_{N}$. Suppose $\sigma:\Ob_N\to\cO_N$ is a $T$-equivariant cosection. Let $M=(N)^{T}$ be the fixed locus, let $\iota$ denote its inclusion in $N$. The fixed and moving parts of $\bE_N$ give respectively a perfect obstruction theory and a virtual normal bundle for $M$.
\[
\bE_M:=\bE_N^{\fix}|_M\qquad N^{\vir}:=(\bE_N^{\mov}|_M)^{\vee}.
\]
It's standard to check that $\bE_M$ is indeed a perfect obstruction theory for $M$. By $T$-equivariance of $\sigma$, it induces a cosection $\sigma':\Ob_M=h^1(\bE_{M})\to\cO_M$ so there is a cosection localized virtual class on $M$. Suppose there is a global resolution
\[
N^{\vir}\cong [N_0\to N_1]
\]
where $N_0$ and $N_1$ are locally-free sheaves on $M$. 
Then Chang--Kiem--Li prove the following virtual torus localization result, extending the result of Graber--Pandharipande \cite{Loc}.
\begin{theorem}\label{cosloc}\cite[Theorem 3.4]{CKL} 

\[
[N]_{\sigma}^\vir=\iota_* \frac{[M]^\vir_{\sigma'}}{e_T(N^{\vir})}\in A^T_*(M)\otimes_{\bbQ}\bbQ[t,t^{-1}].
\]
\end{theorem}

\section{Classical Version}
In this section, $X$ is a smooth projective variety. 
\subsection{The pair $(X,E)$}
Let $E\to X$ be a rank $r$ vector bundle over $X$. Let $P_E\to X$ be the frame bundle of $E$, which is a principal
$GL_r$-bundle over $X$. Then $X =[ P_E/GL_r]$ where $GL_r$ acts freely on the right on $P_E$, and 
$E = P_E\times_{\rho} \bC^r$, where $\rho=\Id_{GL_r}:  GL_r \to GL_r$ is the fundamental representation. Let $\End(E)\cong E\otimes E^\vee$ be the bundle of endomorphisms of $E$. Then $\End(E)= P_E\times_{\Ad}\fgl_r$, where
$\Ad: GL_r \to GL(\fgl_r)$ is the adjoint representation:
$\Ad(g)(\xi)= g^{-1}\xi g$ for $g\in GL_r$ and $\xi\in \fgl_r$. 
We have the following short exact sequence of $GL_r$-equivariant vector bundles over $P_E$:
\begin{equation}\label{eqn:TP}
0\to T_{P_E/X} \cong P_E\times \fgl_r \lra T_{P_E} \lra (\pi_{P_E/X})^*T_X\to 0,
\end{equation}
where $GL_r$-action on $P_E\times \fgl_r$ is given by 
$(e,\xi)\cdot g = (e\cdot g, \Ad(g^{-1})(\xi))$, and $\pi_{P_E/X}: P_E\to X$ is the projection. Taking the quotient of \eqref{eqn:TP} by the
free $GL_r$-action, we obtain the following short exact sequence of vector bundles on $X$:
\begin{equation}\label{eqn:AE}
0 \to \End(E)  \to A_E \to T_X \to 0.
\end{equation}
The vector bundle $A_E$ defines a class 
$$
[A_E] \in \Ext^1(T_X, \End(E)) = H^1(X, \Omega^1_X \otimes \End(E)),
$$ 
known as the Atiyah class of $E$ introduced in \cite{A57}.  Recall that $H^0(X, A_E)$ records the infinitesimal automorphisms of the pair $(X,E)$ and $H^1(X,A_E)$ the inifinitesimal deformations of the pair.
\begin{example}
Let $X=\bbP^m$ and $E=\ccO(k)$, then $P_E=\bC^{m+1}\setminus\{0\}$ has $\bC^*$-action with weights $(-k,\dots ,-k)$. The class of the extension in\cref{eqn:AE} in $[A_{\ccO(k)}]=k[A_{\ccO(1)}]$, where the latter is just the Euler sequence:
\[
[A_{\ccO(1)}]=[0\to\ccO\to\ccO(1)^{\oplus{m+1}}\to T_{\bP^m}\to 0].
\]
\end{example}

\subsection{A section $s$ of $E$.}
Let $s$ be a section of $\pi_{E/X}:E\to X$. This gives rise to a $GL_r$-equivariant map $\phi_s: P_E\to\bbC^r$ that fits in the Cartesian diagram
\begin{equation}
\xymatrix{
P_E \ar[d]\ar[r]^{\Id_{P_E}\times \phi_{s}} & P_E\times \bC^r\ar[d]\\
X=P_E/GL_r \ar[r]^{s\qquad} &\Vb(E) = (P_E\times \bC^r)/GL_r
}
\end{equation}
The map $\phi_{s}: P_E\to \bC^r$ is $GL_r$-equivariant: if $e\in P_E$, $v\in \bC^r$, and $g\in GL_r$, then
$$
\phi_{s}(e\cdot g) = \rho_0(g^{-1}) (\phi_{s}(e)). 
$$
Taking the differential of $\phi_{s}$, we obtain a $GL_r$-equivariant map
$$
T_{P_E} \stackrel{d\phi_{s}}{\lra} \phi_s^* (T_{\bC^r}) =P_E \times \bC^r.  
$$
Taking quotients by the free $GL_r$-action, we obtain
\begin{equation}\label{eqn:deltas}
A_E \stackrel{\delta s}{\lra} E.
\end{equation}

To summarize: a section $s\in H^0(X,E)$ defines an element $\delta s \in \Hom(A_E,E) = H^0(X, A_E^\vee\otimes  E)$.
\begin{example}
Let $X=\bbP^m$, $E=\ccO(k)$. We have established that $A_{\ccO(k)}\cong \ccO(1)^{m+1}$. Given $s$ some homogeneous degree $k$ polynomial in $m+1$ variables, \[
\delta s\in\Hom(\ccO(1)^{m+1},\ccO(k))= H^0(\bP^m,\ccO(k-1)^{\oplus m+1})
\]
 is given by the partial derivatives of $s$. 
\end{example}
Let 
$$
ds: T_X\lra s^*T_E
$$
be the differential of $s$. Let $Z=Z(s)$ be the zero locus of $s$, which is a closed subscheme of $X$. Let $i: Z\hookrightarrow X$ be the inclusion which is 
a closed embedding.
The pullback $i^*s^*T_E$ is a direct sum:
$$
i^*s^*T_E =  i^*T_X \oplus i^*E.
$$  
Let 
\begin{equation}\label{eqn:tangent}
i^*T_X \stackrel{Ds}{\lra} i^*E
\end{equation}
be the composition
$$
i^*T_X \stackrel{i^*ds}{\lra} i^*s^* T_E  =   i^*TX \oplus i^*E \lra i^*E,
$$ 
where the last arrow is projection to the second factor. 
\subsection{Regularity} 
We now assume that $s$ is a {\em regular} section in the sense of \cite[Section 14.1]{Ful}, i.e. the $r$ functions locally defining $Z$ form a regular sequence.
Then $Z$ is a smooth subscheme of $X$ of codimension $r$, and $Ds:i^*T_X\to i^*E$ is a surjective vector bundle map.
We have the following two short exact sequences of vector bundles over $Z$ (which are dual to each other):
$$
0\to T_Z \stackrel{di}{\lra} i^*T_X \stackrel{Ds}{\lra} i^*E = N_{Z/X} \to 0,
$$
$$
0 \to  i^*E^\vee = N^*_{Z/X} \stackrel{(Ds)^*}{\lra} i^* \Omega_X \lra \Omega_Z \to 0. 
$$
In particular,
$$
\bL_Z =  [ i^*E^\vee \stackrel{(Ds)^*}{\lra} i^*\Omega_X ] \simeq  [ 0\to \Omega_Z].  
$$
The condition that the section $s: X\to E$ is regular is equivalent to the condition that $0$ is a regular value of the 
$\bC^r$-valued function $\phi_{s}: P_E\to \bC^r$. It also implies that
$$
\delta s: A_E\to E
$$  
is a surjective vector bundle map.  

\subsection{Superpotential}
The section $s$ defines a regular function $W: \Vb(E^\vee) \to \bC$ on the total space of $E^\vee$, as follows.
Any point in $\Vb(E^\vee)$ is a pair $(x,p)$ where $x$ is a point in $X$ and $p$ is a vector in $E_x^\vee$, 
the fiber of $E\to X$ over $x$; $s(x)$ is a vector in $E_x$. Then
$$
W(x,p) = \langle p, s(x)\rangle
$$  
where $\langle , \rangle$ is the natual pairing between dual vector spaces.
Let $\tW:P_E\times\bC^r\to \bC$ be the composition
$$
P_E\times \bC^r \lra E^\vee = P_E\times_{\rho\vee}\bC^r \stackrel{W}{\lra} \bC.
$$
Then for any $e\in P_E$, $u\in \bC^r$ (viewed a $r\times 1$ matrix), and $g\in GL_r$ (viewed as an $r\times r$ matrix), 
$$
\tW(e, u) = u^t \phi_{s}(e),\quad \tW(e\cdot g, \rho_0^\vee(g^{-1})(u)) = \tW(e\cdot g,g^tu) = u^t g g^{-1} \phi_{s}(e) = \tW(e,u),  
$$
where $u^t \phi_{s}(e)$ is the product of the $1\times r$ matrix $u^t$ and the $r\times 1$ matrix $\phi_{s}(e)\in \bC^r$.
The differential of $\tW$ at $(e,u)\in P_E\times \bC^r$ is
$$
(d\tW)_{(e,u)}:T_eP_E \times \bC^r \lra \bC,\quad (\dot{e}, \dot{u})\mapsto  u^t (d\phi_{s})_e(\dot{e}) + (\dot{u})^t \phi_s(e).
$$
The critical point of $\tW$ is 
$$
\Crit(\tW)=\{ (e,u)\in P_E\times \bC^r:  \phi_{s}(e)=0,\quad {u}\in \left(\mathrm{Im}(d\phi_{s})_e)\right)^\perp \} \subset \phi_{s}^{-1}(0) \times \bC^r.
$$
If $s$ is a regular section, then $\mathrm{Im}\Big((d\phi_{s})_e\Big)=\bC^r$ for all $e\in \phi_{s}^{-1}(0)$, so 
$$
\Crit(\tW) = \phi_{s}^{-1}(0) \times \{0\} \subset P_E\times \bC^r.
$$
and
$$
\Crit(W) =\{ (x,p)\in E^\vee: s(x)=0, p=0\} = Z \subset X \subset E^\vee,
$$
where $X\subset E^\vee$ is the inclusion of the zero section.

\subsection{Classifying spaces}\label{class_spaces}
Let $BGL_r =[ \bullet/GL_r]$ be the classifying space of principal $GL_r$ bundles, or equivalently, the classifying space of rank $r$ vector bundles. 
The projection $\bullet \to [\bullet/GL_r]$ is the universal principal $GL_r$ bundle, and 
the projection $\pi: U_r:= [\bC^r/GL_r]\lra BGL_r= [\bullet/GL_r]$ is the universal rank $r$ vector bundle, where 
the $GL_r$ action on $\bC^r$ is given by $v\cdot g = \rho_0(g^{-1})v$. The vector bundle $E\to X$ defines a morphism
$\psi_E: X\to BGL_r$ such that the following diagram is Cartesian:

\[
\xymatrix{
P_E\ar[d]\ar[r] & \bullet\ar[d]\\
X\ar[r]^{\psi_E}& [\bullet/GL_r]
}
\] 
The morphism $\psi_E$ induces a distinguished triangle
\[
\psi_E^*\bbL_{BGL_r}\to\bbL_X\to \bL_{X/BGL_r}\to \psi_E^*\bbL_{BGL_r}[1].
\]
$\bbL_{BGL_r}$ is quasi-isomorphic to $\mathfrak{gl}_r$, the trivial rank $r$ bundle over a point carrying the adjoint representation of $GL_r$, in degree 1. Then 
\[ \psi_E^*\bL_{BGL_r}\simeq [0\to P_E\times_{\Ad}\mathfrak{gl}_r].
\]
 Since $\psi_E$ is a smooth representable morphism, the above distinguished triangle induces the short exact sequence
\[
0\to \Omega_X\to \Omega_{X/BGL_r}\to P_E\times_{\mbox{Ad}}\mathfrak{gl}_r\to 0 
\]
dual to \cref{eqn:AE}. In other words, the bundle $A_E$ is the dual of the cotangent bundle of the morphism $X\to BGL_r$ induced by the bundle $E$.

 The section $s: X\to E$ defines a morphism $\psi_s: X\to [\bC^r/GL_r]$
such that  $\pi\circ \psi_s =\psi_E$.  More explicitly, we have the following cartesian diagram:
\begin{equation}\label{eqn:XEs}
\xymatrix{
E  =(P_E\times \bC^r)/GL_r   \ar[d]^{\pi_{E/X}} \ar[r]^{\widetilde{\psi_E}}  &\qquad U_r=[\bC^r/GL_r]\ar[d]^\pi  \\
X  =P_E/GL_r \qquad \ar@/^/[u]^s \ar[r]^{\psi_E} \ar[ur]^{\psi_s}  & BGL_r=[\bullet/GL_r] 
}
\end{equation}
The composition  $\pi_{E/X}\circ s:X\to X$ is the identity map since
$s:X\to E$ is a section of $\pi_{E/X}:E\to X$, and $\psi_s = \widetilde{\psi_E}\circ s: X\to [\bC^r/GL_r]$. 
Therefore, $[\bC^r/GL_r]$ can be viewed as the classifying space of pairs $(E,s)$, where $E$ is a rank $r$ vector bundle
and $s$ is a section of $E$.

Diagram \eqref{eqn:XEs} is obtained by taking quotients by the $GL_r$-actions on the following 
cartesian diagram of $GL_r$-spaces (where all the arrows are $GL_r$-equivariant):
\begin{equation}\label{eqn:ts} 
\xymatrix{
P_E\times \bC^r   \ar[d]^{\pr_1} \ar[r]^{\pr_2}  & \bC^r\ar[d] \\
P_E  \ar@/^/[u]^{\Id_{P_E}\times \phi_{s}} \ar[r] \ar[ur]^{\phi_{s}}  & \bullet 
}
\end{equation}
where $\pr_1: P_E\times \bC^r \to P_E$ and $\pr_2: P_E\times \bC^r \to \bC^r$ are projections to the first and second factors, respectively.
The map $\delta s$ of \cref{eqn:deltas} is induced by the differential of $\psi_s$ in \cref{eqn:XEs}:
\begin{equation}\label{deltas}
\bbT_{X/BGL_r}=A_E\to \psi_s^*\bbT_{U_r/BGL_r}\cong s^*T_{Vb(E)/X}\cong E.
\end{equation}
Since $s$ is a regular section, $\delta s$ is surjective.
\section{Quantum Version}\label{sect: quantum}

We will fix a triple $(X,E,s)$ where $X$ is a smooth projective variety, $\pi_{E/X}:E\to X$ is a vector bundle, and
$s:X\to E$ is a regular section. We will denote a morphism $M_1\to M_2$ between moduli stacks $\pi_{M_1/M_2}$. 

\subsection{Moduli stacks independent of $(X,E,s)$}\label{ss: SandP}
Let $\fM=\fM_{g,n}$ be the moduli stack of genus $g$, $n$-pointed prestable curves. Then $\fM$ is a smooth Artin stack of dimension $3g-3+n$.
Let $\fC_{\fM}\to \fM$ be the universal curve.

The analog of $BGL_r$ is the moduli space $
\Bun_{g,n,r}$ parametrizing prestable genus $g$ $n$-pointed curves together with a rank $r$ vector bundle. This can be viewed as a moduli of sections of 
\[ \Bun_{g,n,r}=\Gamma(\fC_{\fM}\times BGL_r/\fC_{\fM}).\] 
For a scheme $T\xrightarrow{f} \fM$ denote pullback of the universal curve by $C_T$. Objects in the fiber of $\Bun_{g,n,r}$ over $T\to \fM$ are sections of ${f}^*\fC_{\fM}\times BGL_r = C_T\times BGL_r\to C_T$. That is, morphisms $C_T\to BGL_r$. So the fiber of $\Bun_{g,n,r}$ over $T/\fM$ consists of rank $r$ vector bundles over $C_T$. 
As shown in \cite{ciocan2014stable}, the moduli space $\Bun_{g,n,r}$ is itself a smooth Artin stack of dimension $(3+r^2)(g-1)+n$. The degree of the restriction of a vector bundle $F_T$ over $C_T$ to the fiber over a point $t\in T$ is constant for connected schemes $T$ since $C_T\to T$ is a flat projective family. We further fix $d\in \bZ$ and let
$$
\fB= \Bun_{g,n,r,d} 
$$
be the connected component of $\Bun_{g,n,r}$ corresponding to vector bundles which have degree $d$. Objects in $\fB(\bullet)$ are pairs $((C,x_1,\ldots,x_n),F)$ where $(C,x_1,\ldots,x_n)$ is a pointed prestable curve and $F$ a degree $d$ rank $r$ vector bundle. 

The morphism $\pi_{\fB/\fM}:\fB \lra \fM$ given by forgetting the vector bundle  is smooth of relative dimension $r^2(g-1)$. Let $\fC_{\fB}=\fC_{\fM}\times_{\fM}\fB$ be the universal curve on $\fB$, $\cP_{\fB}\to\fC_{\fB}$ the universal $GL_r$-torsor and $\cE_{\fB}$ be the associated universal vector bundle.

We can further form a cone stack $\fS$ over $\fB$ parametrizing curves with a vector bundle and a section:
$$
\fS =\Spec(\Sym \;R^1\pi_{\fB*}\cE_{\fB}^{\vee}\otimes\omega_{\pi_{\fB}}).
$$
For any scheme $T$ and morphism $T\to \fB$, $\fS(T)=H^0(C_T,\cE_T)$ by Serre duality.  Here $\cE_T\to C_T\to T$ is the pullback of the universal family $\cE_{\fB}\to \fC_{\fB}\to \fB$.
Similarly, we have $\fS_{g,n,r}\to\fB_{g,n,r}$.

By \cite{CL} Lemma 2.5, there is a relative dual perfect obstruction theory $\bE_{\fS/\fB}^{\vee}$ given by 
$$
\bE^{\vee}_{\fS/\fB} = \pi^*_{\fS/\fB} R^\bullet \pi_{\fB*} \cE_{\fB}.
$$
The relative virtual dimension is 
$$
d^\vir_{\fS/\fB} = h^0(C,V)-h^1(C,V)  = d + r(1-g)
$$
where $((C,x_1,\ldots,x_n), V)$ is any object in $\fB(\bullet)$. 

Finally, we consider the moduli 
$$
\fP =  \Spec_{\fB}\Sym \left(R^1\pi_{\fB*} \big( \cE_{\fB} \oplus \cE_{\fB}^\vee\otimes \omega_{\fC_{\pi_{\fB}}}\big)\right)
$$
The objects in the groupoid $\fP(\bullet)$ are 4-tuples $((C,x_1,\ldots,x_n),V,q,p)$ where
$((C,x_1,\ldots,x_n),V,q)$ is an object in $\fS(\bullet)$ and $p \in H^0(C, V^\vee\otimes \omega_C)$. There is
a morphism $\pi_{\fP/\fS}:\fP\to\fS$ given by forgetting $p$. There is a relative dual perfect theory $\bE^\bullet_{\fP/\fS}$ given by
$$
\bE^{\vee}_{\fP/\fS} = \pi^*_{\fP/\fB}R^\bullet\pi_* \left(\cE_{\fB}^\vee \otimes \omega_{\fC_{\fB}/\fB}\right).
$$
The relative virtual dimension is
$$
d^\vir_{\fP/\fS}=  h^0(C,V^\vee\otimes \omega_C) - h^1(V^\vee\otimes \omega_C) =  -d + r(g-1).
$$
The relative virtual dimension of $\pi_{\fP/\fB}:\fP\to \fB$ is $d^\vir_{\fP/\fB} =0$.  A perfect dual obstruction theory of $\pi_{\fP/\fB}$ is given by
\[
\bE^{\vee}_{\fP/\fB}  = \pi^*_{\fP/\fB}R^{\bullet}\pi_{\fB*} \left(\cE_{\fB}\oplus \cE_{\fB}^\vee \otimes \omega_{\pi_{\fB}}\right)
\]
and is compatible with $\bE^{\vee}_{\fP/\fS}$ and $\bE^{\vee}_{\fS/\fB}$.

We now define a relative cosection
\[
\si_{\fP/\fB}: \Ob_{\fP/\fB}\lra   \cO_{\fP}
\]
as follows. We have a map of vector bundles over $\fC_{\fB}$ 
\[
\cE_{\fB}\oplus \cE_{\fB}^{\vee}\otimes\omega_{\pi_{\fB}}\to \omega_{\pi_{\fB}}
\]
given by the dual pairing. Applying $\pi_{\fP/\fB}^*R^1\pi_{\fB *}$ to this map induces the relative cosection.
Given an object $\xi=((C,x_1,\ldots,x_n),F,q,p)$ in $\fP(\bullet)$, 
\begin{align*}
\si_{\fP/\fB}(\xi) : H^1(C,F)\oplus H^1(C,F^\vee\otimes \omega_C)& \lra H^1(C,\omega_C) =\bC\\
(\dot{q},\dot{p}) &\mapsto \langle  \dot{p},q\rangle + \langle p,\dot{q}\rangle
\end{align*}
where $\dot{q}\in H^1(C,F)$ and $\dot{p}\in H^1(C,F^\vee\otimes \omega_C)$ and the pairing is given by Serre duality. Then $\si_{\fP/\fB}(\xi)=0$ if and only if $q=p=0$. The degeneracy locus of the relative cosection is the zero section of the cone $\fP\to \fB$: 
\[
D(\si_{\fP/\fB} )= \{ (\si_{\fP/\fB})(\xi)=0\} = \fB \subset \fP.
\]
Unfortunately, even though the morphism $\fP\to\fB$ is representable, we cannot use the relative perfect obstruction theory to construct a cosection-localized virtual class on $\fP$ as the construction in \cite{KL} would require the stack $\fP$ to be a DM stack itself. 

\subsection{Moduli of maps to $X$}\label{ss: potX}
The triple $(X,E,s)$ gives rise to the following commutative diagram
\begin{equation}\label{eqn:cXSB}
\xymatrix{
X \ar[r]^{\psi_s} \ar[dr]^{\psi_E} & [\bC^r/GL_r] \ar[d]\\
& BGL_r =[\bullet/GL_r]
}
\end{equation}
The moduli of sections $\fC_{\fM}\times X \to \fC_{\fM}$ is the moduli of genus $g$, $n$-pointed prestable maps to $X$:
$$
\fX_{g,n} =\Gamma((\fC_{\fM}\times X)/\fC_{\fM}).
$$
The quantum version of \eqref{eqn:cXSB} is
\begin{equation}\label{eqn:qXSB-I}
\xymatrix{
\fX_{g,n} = \Gamma((\fC_{\fM}\times X)/\fC_{\fM}) \ar[r] \ar[dr]  & \fS_{g,n,r}\ar[d]\\
& \Bun_{g,n,r}\ar[d] \\
&\fM= \fM_{g,n} 
}
\end{equation}
We now fix an effective curve class $\beta\in H_2(X;\bZ)$, and let $d= \langle c_1(E),\beta\rangle\in \bZ$.  Let $\fX_{g,n,\beta}$ be the moduli stack
of genus $g$, $n$-pointed, degree $\beta$ prestable maps to $X$. Then $\fX_{g,n,\beta}$ is an open and closed substack of the Artin stack
$\fX_{g,n}$, and the above diagram \eqref{eqn:qXSB-I} restricts to 
\begin{equation}\label{eqn:qXSB-II}
\xymatrix{
 \fX_{g,n,\beta}  \ar[r] \ar[dr] & \fS  \ar[d]\\
& \fB= \Bun_{g,n,r,d}\ar[d] \\
& \fM= \fM_{g,n}
}
\end{equation}
Let $\fX$ denote the open substack of  $\fX_{g,n,\beta}$ obtained by imposing the usual stability conditions for stable maps. The functor $\fX(\bullet) \lra \fS(\bullet)$ sends $((C,x_1,\ldots,x_n), f)$ to $(C,f^*E, f^*s)$.

\subsection{A perfect obstruction theory for $\pi_{\fX/\fB}: \fX\to \fB$}
Let $\pi_{\fX}: \fC_{\fX}\to \fX$ be the universal curve over $\fX$, and let $ev_{\fX}: \fC_{\fX}\to X$ be the universal evaluation map. The morphism $\fX\to\fB$ over $\fM$ is induced by fixing the torsor $ev_{\fX}^*P_E\to\fC_{\fX}$, which we denote $\cP_{\fX}$. Recall $\cP_{\fB}$ denotes the universal $GL_r$-torsor over $\fC_{\fB}$. We have the following diagram, where all the squares are Cartesian. 
\[
\xymatrix{
P_E\ar[d]& \ar[l] \cP_{\fX}\ar[r] \ar[d] & \cP_{\fB}\ar[d] \\
X & \ar[l]_{ev_\fX}\fC_{\fX}\ar[r]^{\overline{\pi_{\fX/\fB}}} \ar[d]^{\pi_{\fX}}& \fC_{\fB}\ar[d] \\
& \fX\ar[r]^{\pi_{\fX/\fB}} & \fB 
}.
\]

Having fixed these $GL_r$-torsors, we have the following commutative diagram

\begin{equation}\label{Tor1}
\xymatrix{
\fC_{\fX}\ar[d]_{\overline{\pi_{\fX/\fB}}}\ar[r]^{ev_{\fX}} & X\ar[d]^{\psi_E}\\
\fC_{\fB}\ar[r]^{ev_{\fB}} & BGL_r
}
\end{equation}
which induces a morphism of cotangent complexes
\[
ev_{\fX}^*\bL_{X/BGL_r} \to\bL_{\fC_{\fX}/\fC_{\fB}}=\pi_{\fX}^*\bL_{\fX/\fB}.
\]
Dualizing and pushing forward to we obtain:
\[
\phi_{\fX/\fB}^\vee: \bT_{\fX/\fB}\to R^{\bullet}\pi_{\fX*}\pi_{\fX}^*\bT_{{\fX}/{\fB}}\to R^{\bullet}\pi_{\fX*}ev_{\fX}^*A_E.
\]
 
\begin{proposition}\label{potX}
The morphism 
\[
\phi_{\fX/\fB}:  \bE_{\fX/\fB}=(R^{\bullet}\pi_{\fX*}ev_{\fX}^*A_E)^{\vee}\to\bL_{\fX/\fB}.
\]
is a perfect obstruction theory for $\pi_{\fX/\fB}$. 
\end{proposition}
\begin{proof}
We can see this by identifying $\fX$ with an open substack of the moduli of sections $\Gamma\coloneqq \Gamma(\fC_{\fB}\times_{BGL_r} X/ \fC_{\fB})$. To see this, fix an object $\xi=((C,x_1,\dots ,x_n),F)\in\fB(\bullet)$. An object $\overline{\xi}\in\fX_{g,n,\beta}(\bullet)$ (recall these are prestable maps to $X$) in the fiber of $\xi$ is the additional datum of a morphism $f:C\to X$ such that $f^*E\cong F$. Let $\psi_F:C\to BGL_r$ be the morphism that $F$ represents. Consider the following Cartesian diagram
\[
\xymatrix{
C\times_{BGL_r}X\ar[r]\ar[d]& X\ar[d]^{\psi_E}\\
C\ar[r]_{\psi_F} & BGL_r.
}
\]
A morphism $f:C\to X$ such that $\psi_E\circ f=\psi_F$ is the same as a a section $\phi_f$ of the projection $C\times_{BGL_r}X\to C$ above. The same argument applies for objects $\xi\in\fB(T)$ over general schemes $T$, so this gives the required identification.
Observe that $\fC_{\fB}\times_{BGL_r} X\cong \cP_{\fB}\times_{GL_r} P_E$, so $\fC_{\fB}\times_{BGL_r} X\to\fC_{\fB}$ is a smooth representable morphism, in fact a smooth fibration with fiber $P_E$. Then the following is a perfect obstruction theory for $\Gamma\to\fB$
\[
\phi_{\Gamma/\fB}:\bE_{\Gamma/\fB}\coloneqq (R^\bullet\pi_{\Gamma *}\mathfrak{e}^*\Omega^\vee_{\fC_{\fB}\times_{BGL_r} X/\fC_{\fB}})^\vee \to \bL_{\Gamma/\fB},
\]
where $\mathfrak{e}:\fC_{\Gamma}\to \fC_{\fB}\times_{BGL_r} X$ is the universal evaluation. Over $\fC_{\fX}$, $\mathfrak{e}$ restricts to the map $\mathfrak{f}: \fC_{\fX}\to \fC_{\fB}\times_{BGL_r} X$ induced by the diagram in \cref{Tor1}. Then
\begin{align*}
\pi_{\fX/\Gamma}^* (R^\bullet\pi_{\Gamma *}\mathfrak{e}^*\Omega^\vee_{\fC_{\fB}\times_{BGL_r} X/\fC_{\fB}})^\vee & \cong (R^\bullet\pi_{\fX*}\overline{\pi_{\fX/\Gamma}}^*\mathfrak{e}^*\Omega^\vee_{\fC_{\fB}\times_{BGL_r} X/\fC_{\fB}})^\vee  \\
&  \cong (R^\bullet\pi_{\fX*}\mathfrak{f}^*\Omega^\vee_{\fC_{\fB}\times_{BGL_r} X/\fC_{\fB}})^\vee \\
& \cong (R^\bullet\pi_{\fX*}ev_{\fX}^*\Omega^\vee_{X/BGL_r})^\vee.
\end{align*}
So $\phi_{\Gamma/\fB}$ restricts to $\phi_{\fX/\fB}$ over $\fX$, which proves that the latter is a perfect obstruction theory.
\end{proof}

Let 
\[
\phi_{\fX/\fM}:  \bE_{\fX/\fM}=(R^{\bullet}\pi_{\fX*}ev_{\fX}^*T_X)^{\vee}\to \bL_{\fX/\fM}
\]
 denote the usual perfect obstruction theory for $\pi_{\fX/\fM}:\fX\to\fM$. Let $(C,f)\in\fX(\bullet)$, the virtual dimension of $\fX/\fM$ given by $\phi_{\fX/\fM}$ is 
 \[
 d^{\vir}_{\fX/\fM}=\chi(f^*T_X)=\langle \beta, c_1(X)\rangle + dim(X)(1-g)
 \]
 and the virtual dimension of the obstruction theory above is 
 \[ d^{\vir}_{\fX/\fB}=\chi(f^*A_E)=\chi(f^*T_X)+\chi(f^*(E\otimes E^{\vee}))= d^{\vir}_{\fX/\fM}-d_{\fB/\fM}
 \]
 by the short exact sequence in \cref{eqn:AE}. So the virtual fundamental classes induced by the two obstruction theories have the same degree $\langle \beta, c_1(X)\rangle+(dim(X)-3)(1-g)+n=:d^{\vir}_{\fX}$.

\begin{proposition}\label{potXcompare}
The virtual fundamental class $[\fX/\fB]^{\vir}$ induced by $\phi_{\fX/\fB}$ agrees with $[\fX/\fM]^{\vir}$ induced by $\phi_{\fX/\fM}$.
\end{proposition}
\begin{proof}
 This follows closely \cite{CL} Propositions 2.8 and 2.9.
We have the following commutative diagram of evaluation morphisms
\begin{equation}\label{comp_evals}
\xymatrix{
\fC_{\fX}\ar[d]_{(\overline{\pi_{\fX/\fM}},ev_\fX)} \ar[r]^{ \overline{\pi_{\fX/\fB}}}& \fC_{\fB}\ar[d]_{(\overline{\pi_{\fX/\fB}},ev_\fB)} \ar[r]^{\overline{\pi_{\fB/\fM}}} & \fC_{\fM}\ar[d]_{id}\\
\fC_{\fM}\times X\ar[r]_{(id, \psi_E)} & \fC_{\fM}\times BGL_r\ar[r]_{\qquad pr_1} & \fC_{\fM} 
}
\end{equation}
The first row of \cref{comp_evals} induces the distinguished triangle
\begin{equation}\label{triangleLXBM}
\xymatrix{
\pi_\fX^*\bL_{\fX/\fM} \ar[d]^{\cong}\ar[r]& \pi_\fX^*\bL_{\fX/\fB} \ar[d]^{\cong}\ar[r]& \overline{\pi_{\fX/\fB}}^*\pi_{\fB}^*\bL_{\fB/\fM}[1]\ar[d]^{\cong}\ar[r]^{\qquad\qquad+1}& \\
\bL_{\fC_{\fX}/\fC_{\fM}}\ar[r]& \bL_{\fC_{\fX}/\fC_\fB}\ar[r] & \overline{\pi_{\fX/\fB}}^*\bL_{\fC_\fB/\fC_\fM}[1]\ar[r]^{\qquad\qquad+1}& \\
}
\end{equation}
The second row of \cref{comp_evals} induces
\[
\bL_{\fC_{\fM}\times X/\fC_{\fM}}\to 	\bL_{\fC_{\fM}\times X/\fC_{\fM}\times BGL_r}\to (id,\psi_E)^*\bL_{\fC_{\fM}\times BGL_r/\fC_{\fM}}[1]\xrightarrow{+1}
\]
which, after pulling back to $\fC_{\fX}$ via $(\overline{\pi_{\fX/\fM}},ev_\fX)$, is isomorphic to

\begin{equation}\label{triangleEXBM}
ev_{\fX}^*\bL_{X}\to 	ev_{\fX}^*\bL_{X/BGL_r}\to ev_{\fX}^*\psi_E^*\bL_{BGL_r}[1]\xrightarrow{+1}
\end{equation}
As we detailed in \cref{class_spaces} the above triangle is quasi-isomorphic to
\[
0\to ev_{\fX}^*T_X^\vee \to ev_{\fX}^*A_E^{\vee}\to ev_{\fX}^*(\End(E))^\vee\to 0.
\]
There is a morphism of distinguished triangles between \cref{triangleEXBM} and \cref{triangleLXBM} induced by the differential of the evaluation map. After dualizing, pushing forward by $\pi_{\fX}$ and dualizing again we obtain:  
\begin{equation}\label{XBMtriangle}
\xymatrix{
\bE_{\fX/\fM}\ar[r]\ar[d]^{\phi_{\fX/\fM}}& \bE_{\fX/\fB}\ar[r]\ar[d]^{\phi_{\fX/\fB}} &  (R^\bullet\pi_{\fX *}\End(\cE_{\fX}))^\vee\ar[d]\ar[r]^{\qquad\qquad+1}& \\
\bL_{\fX/\fM} \ar[r]& \bL_{\fX/\fB}\ar[r]& {\pi_{\fX/\fB}}^*\bL_{\fB/\fM}[1]\ar[r]^{\qquad\qquad+1}& .
}
\end{equation}
The last vertical arrow is a quasi-isomorphism, noting that 
\[
(R^\bullet\pi_{\fX *}\End(\cE_{\fX}))^\vee\cong {\pi_{\fX/\fB}}^*(R^\bullet\pi_{\fB *}\End(\cE_{\fB}))^\vee
\]
by cohomology and base change and that $\bL_{\fB/\fM}[1]\cong (R^\bullet\pi_{\fB *}\End(\cE_{\fB}))^\vee$. Let
 $\ffN$ denote the vector bundle stack $\cs(R^\bullet\pi_{\fB *}\End(\cE_{\fB}))$. By \cref{conelemma} below, the bottom triangle of \cref{XBMtriangle} induces a short exact sequence of cone stacks 
 \[
 \pi^*_{\fX/\fB}\ffN\to \ccC_{\fX/\fB}\to \ccC_{\fX/\fM}
 \]
 where $\ccC_{\fX/\fB}$ is the intrinsic normal cone of the morphism $\pi_{\fX/\fB}$, and similarly for $\ccC_{\fX/\fM}$. The top row of \cref{XBMtriangle} gives a similar short exact sequence of vector bundle stacks, and we have
 \begin{equation*}
 \xymatrix{
  \pi^*_{\fX/\fB}\ffN\ar[r]\ar[d]^{\cong}& \ccC_{\fX/\fB}\ar[r]\ar@{^{(}->}[d]& \ccC_{\fX/\fM}\ar@{^{(}->}[d]\\
 \pi^*_{\fX/\fB}\ffN \ar[r]&  E_{\fX/\fB}\ar[r] &  E_{\fX/\fM}}
 \end{equation*}
  where $E_{\fX/\fB}=\cs (\bE_{\fX/\fB})$ and similarly for $E_{\fX/\fM}$. Recall that by definition short exact sequences of cone stacks are locally split. Then \[
  [\fX/\fB]^{\vir}=0^!_{E_{\fX/\fB}}([\ccC_{\fX/\fB}])=0^!_{E_{\fX/\fM}}([\ccC_{\fX/\fM}])=[\fX/\fM]^{\vir}.
  \]
  
 \end{proof}
 \begin{lemma}\label{conelemma}
 Let 
$$
\fX \stackrel{\pi_{\fX/\fB}}{\lra} \fB \stackrel{\pi_{\fB/\fM}} {\lra} \fM.
$$
be morphisms of Artin stacks. Assume that $\fM$ is smooth and $\pi_{\fB/\fM}$ is smooth. (So $\fB$ is smooth.)  Let $\pi_{\fX/\fM} = \pi_{\fB/\fM}\circ \pi_{\fX/\fB}$.
Assume that $\pi_{\fX/\fB}$ and $\pi_{\fX/\fM}$ are of DM type. Let $\ccC_{\fX/\fB}$ and $\ccC_{\fX/\fM}$ denote the relative intrinsic normal cones.
Then we have the following short exact sequence of cone stacks over $\fX$:
$$
h^1/h^0(\pi_{\fX/\fB}^* \left(\bL_{\fB/\fM}[1]\right)^\vee ) \to \ccC_{\fX/\fB} \to \ccC_{\fX/\fM}.
$$
 \end{lemma}
 \begin{proof}
We have a distinguished triangle
\[
\pi_{\fX/\fB}^*\bT_{\fB/\fM}[-1]\to\bT_{\fX/\fB}\to\bT_{\fX/\fM}
\]
so a short exact sequence
\[
h^1/h^0(\pi_{\fX/\fB}^*\bT_{\fB/\fM}[-1])\to \fN_{\fX/\fB}\to\fN_{\fX/\fM}.
\]
where $ \fN_{\fX/\fB}$ and $ \fN_{\fX/\fM}$ denote the relative intrinsic normal sheaves. 

The intrinsic normal cone $\ccC_{\fX/\fB}$ is the unique subcone of $\fN_{\fX/\fB}$ such that $\ccC_{\fX/\fB}\times_\fX U\cong[C_{U/Y}/i^*T_{Y/\fB}]$ for any local embedding of $\fX$ into $\fB$ given by a commutative diagram
\begin{equation}\label{UYB}
\xymatrix{
U\ar[r]^{i} \ar[d]_{p}&Y\ar[d]\\
\fX\ar[r]&\fB
}
\end{equation}
where $i$ is a closed embedding, $U$, $Y$ are schemes with $Y$ smooth over $\fB$ and $U$ \'{e}tale over $\fX$. This is Proposition 2.13 of \cite{VPB}. 
Since $\fB\to\fM$ is smooth, $Y\to\fM$ is smooth also and $\ccC_{\fX/\fM}\times_{\fX}U\cong[C_{U/Y}/i^*T_{Y/\fM}]$. On the other hand, given any local embedding of $\fX$ in $\fM$, ie. a diagram similar to \eqref{UYB}: 
\begin{equation}\label{UYM}
\xymatrix{
U'\ar[r]^{j} \ar[d]&Y'\ar[d]\\
\fX\ar[r]&\fM
}
\end{equation}
we can reduce it to \eqref{UYB} by setting $Y=Y'\times_\fM\fB$, $U=U'$ with $U\to Y$ induced by $U'\to\fX\to\fB$ and $U'\to Y'$. So we only have to show the short exact sequence in question holds when restricted to $U$ for any diagram of the form \eqref{UYB}, which follows from the distinguished triangle $p^*\pi^*_{\fX/\fB}\bT_{\fB/\fM}[-1]\to i^*T_{Y/\fB}\to i^*T_{Y/\fM}$.

\end{proof}

\subsection{Moduli of maps to $X$ with fields}\label{ss:fields}
Let
$$
\fX^E_{g,n,\beta} ={\Spec}_{\fX}\Sym \,R^1\pi_{\fX*}(\cE_{\fX})
$$
be the moduli of genus $g$, $n$-pointed, degree $\beta$ maps to $X$ with a $p$-field. Objects in 
the groupoid $\fX^E(\bullet)$ are triples $((C,x_1,\ldots,x_n), f, p)$ where $((C,x_1,\ldots, x_n), f)$ is an object in $\fX(\bullet)$ and
$p\in H^0(C,f^*E^\vee\otimes\omega_C)$. The diagram \eqref{eqn:qXSB-I} is extended to the following diagram:
\begin{equation}\label{qXSBP}
\xymatrix{
\fX^E:=\fX^E_{g,n,\beta} \ar[r] \ar[d] & \fP =\fP_{g,n,r,d}\ar[d]\\
\fX:= \fX_{g,n,\beta}  \ar[r] \ar[dr] & \fS=\fS_{g,n,r,d}  \ar[d]\\
& \fB= \Bun_{g,n,r,d}\ar[d] \\
& \fM= \fM_{g,n}
}
\end{equation}
From its definition, we see that $\fX^E$ has a dual relative perfect obstruction theory over $\fX$ given by 
\[
\bE_{\fX^E/\fX}^\vee = R^\bullet\pi_{\fX^E *}(\cE_{\fX^E}^\vee\otimes\omega_{\pi_{\fX^E}}).
\]

We can identify $\fX^E\to\fB$ with an open substack of the moduli of sections 
\[
\Gamma\left((\fC_{\fB}\times_{BGL_r}X)\times_{\fC_{\fB}}\Vb(\cE_{\fB}^{\vee}\otimes\omega_{\pi_{\fB}})\right).
\]
 Similarly to \cref{potX}, $\pi_{\fX^E/\fB}$ has a dual perfect obstruction theory given by 
\begin{align*}
\bE_{\fX^E/\fB}^{\vee}&=R^\bullet\pi_{\fX^E*}\left(ev^*A_E\oplus\cE_{\fX^E}^\vee\otimes\omega_{\pi_{\fX^E}}\right)
\end{align*}
The morphism $(\phi_{\fX^E/\fB})^{\vee}:\bT_{\fX^E/\fB}\to\bE^{\vee}_{\fX^E/\fB}$ is induced by the differential of the universal evaluation morphisms.
This perfect obstruction theory is just 
\[
\bE_{\fX^E/\fB}^{\vee}=\pi_{\fX^E/\fX}^*\bE_{\fX/\fB}^\vee\oplus R^\bullet\pi_{\fX^E *}(\cE_{\fX^E}^\vee\otimes\omega_{\pi_{\fX^E}}).
\]

In the sections below, we will construct a cosection of the (relative) perfect obstruction theory of $\fX^E$. Since \cref{thm:coslocvir} crucially needs a cosection of the absolute obstruction sheaf, we will need to check that our cosection lifts to a cosection of the absolute obstruction sheaf. 

\subsection{The cosection.}\label{ss: cosection}
The section $s\in H^0(X,E)$ induces $\psi_s:X\to U_r=[\bC^r/GL_r]$ which is a morphism over $BGL_r$. Pulling back by $\fC_{\fB}\to BGL_r$ gives
\[
\xymatrix{
\fC_{\fB}\times_{BGL_r}X\ar[r]^{\;\; \mathfrak{s}}\ar[dr]& \Vb(\cE_{\fB})\ar[d]\\
&\fC_{\fB}.
}
\]
Since $\fX$ and $\fS$ are the moduli of sections of $\fC_{\fB}\times_{BGL_r}X$ and $\Vb(\cE_{\fB})$ respectively, $\mathfrak{s}$ induces both the morphism  $\pi_{\fX/\fS}$ from \cref{eqn:qXSB-II} and a morphism 
\[
\delta\mathfrak{s}:\bE_{\fX/\fB}^\vee= R^\bullet\pi_{\fX *}(ev_{\fX}^*A_E) \to \pi_{\fX/\fS}^*\bE_{\fS/\fB}^\vee=R^\bullet\pi_{\fX *}(\cE_\fX)
\]
which by \cref{deltas} is just $R^{\bullet}\pi_{\fX *}(ev_{\fX}^*\delta s)$ for $\delta s: \bT_{X/BGL_r}\cong A_E\to E\cong \psi_s^*\bT_{U_r/BGL_r}$.
Similarly, the top row of \cref{eqn:qXSB-II} is induced by 
\[
(\mathfrak{s},id): (\fC_{\fB}\times_{BGL_r}X)\times_{\fC_{\fB}}\Vb(\cE_{\fB}^{\vee}\otimes\omega_{\fC_{\fB}/\fB})\to\Vb(\cE_{\fB})\times_{\fC_{\fB}}\Vb(\cE_{\fB}^{\vee}\otimes\omega_{\fC_{\fB}/\fB}).
\]
This also gives 
\[
\bE_{\fX^E/\fB}^{\vee}\to\pi_{\fX^E/\fP}^*\bE^{\vee}_{\fP/\fB}.
\]
Composing the resulting $\Ob_{\fX^E/\fB}\to\pi_{\fX^E/\fP}^*\Ob_{\fP/\fB}$ with the pullback of the cosection $\sigma_{\fP/\fB}$ of  \cref{ss: SandP} gives relative cosection $\sigma_{\fX^E/\fB}:\Ob_{\fX^E/\fB} \lra \cO_{\fX^E}$.

This relative cosection agrees with the one that can be constructed directly from the morphism
\begin{align*}
h_{\fX^E/\fB}:& (\fC_{\fB}\times_{BGL_r}X)\times_{\fC_{\fB}}\Vb(\cE_B^\vee\otimes\omega_{\fC_{\fB}/\fB})\to \Vb(\omega_{\fC_{\fB}/\fB}),\qquad (z,p)\mapsto p\cdot \mathfrak{s}(z).
\end{align*}
By construction of the relative perfect obstruction theory $\phi_{\fX^E/\fB}$, $h_{\fX^E/\fB}$ gives a map of complexes
\begin{equation}\label{cosecXEB}
\sigma_{\fX^E/\fB}^\bullet: \bE_{\fX^E/\fB}^{\vee}\to R^\bullet\pi_{\fX^E*}\omega_{\fC_{\fX^E}/\fX^E}.
\end{equation}
And $\sigma_{\fX^E/\fB}:\Ob_{\fX^E/\fB} \lra \cO_{\fX^E}$ is $h^1(\sigma_{\fX^E/\fB}^\bullet)$.

\begin{proposition}
Let $\xi=[(C,x_1,\dots,x_n),f,p]\in\fX^E$ be a closed point. The cosection above has local expression
\[
\sigma_{\fX^E/\fB}|_\xi: H^1(f^*A_E)\oplus H^1(f^*E^\vee\otimes\omega_C)\to\bC
\]
given by $(\dot{z},\dot{p})\mapsto \langle f^*\delta s(\dot{z}), p\rangle +\langle f^*s,\dot{p}\rangle$ where the brackets are given by Serre's duality. 

\end{proposition}
\begin{proof}
Follows from the description of the cosection $\sigma_{\fP/\fB}$ and the observation that $\delta\mathfrak{s}$ is given by $R^{\bullet}\pi_{\fX *}(ev_{\fX}^*\delta s)$.
We show this explicitly, to illustrate how the cosection arises from the morphism $h_{\fX^E/\fX}$.
Let $C\to BGL_r$ be the morphism induced by $f^*E$.
We have 
\[
h_{\fX^E/\fB}|_\xi: (C\times_{BGL_r}X)\times_C\Vb(f^*E^\vee\otimes\omega_C)\to\Vb(\omega_C).
\]
This is induced by $X\xrightarrow{\psi_s} U_r$, which gives $C\times_{BGL_r}X\to\Vb(f^*E)$ and the pairing between dual vector spaces in the fibers. The evaluation morphism of $C$ into $ (C\times_{BGL_r}X)\times_C\Vb(f^*E^\vee\otimes\omega_C)$ is given by $ev_C=((id_C,f),p)$. Then $ev_C^*h_{\fX^E/\fB}|_{\xi}$ is simply
\[
C\to\Vb(f^*E)\times_C\Vb(f^*E^\vee\otimes\omega_C)\to\Vb(\omega_C)
\] 
with the first arrow given by the sections $f^*s$ and $p$ and the second by fiber-wise dual pairing. 
Following the definition of the cosection, we consider
\begin{equation}\label{dihacca}
ev_C^*(dh_{\fX^E/\fB}|_{\xi}): ev_C^*\bT_{(C\times_{BGL_r}X)\times_C (\Vb(f^*E^\vee\otimes\omega_C)/C}\to ev_C^* h_{\fX^E/\fB}^*\bT_{\Vb(\omega_C)/C}.
\end{equation}
Observe that $(id_C,f)^*\bT_{C\times_{BGL_r}X/C}$ is quasi isomorphic to $f^*A_E$. Recall that the differential of $\psi_s$ is the morphism $\delta s:A_E\to E$ from \cref{deltas}. Then we see that \cref{dihacca} is quasi isomorphic to the morphism of sheaves over $C$
\[
\sigma: f^*A_E\oplus f^*E^\vee\otimes\omega_C\to \omega_C
\]
 given by $\langle f^*(\delta s), p\rangle+\langle\dot{p}, f^*s\rangle$. The brackets are defined by the pairing $\langle\cdot,\cdot\rangle: V\otimes  H^0(C,V^\vee)\to \cO_C$ for a locally free sheaf $V$ on $C$. Since $\sigma_{\fX^E/\fB}|_\xi$ is by definition $H^1(\sigma)$ we are done.
\end{proof}

This local expression allows to easily compute the degeneracy locus
\[
D(\sigma_{\fX^E/\fB})\coloneqq \{ \xi\in\fX^E | \sigma_{\fX^E/\fB}|_\xi:\Ob_{\fX^E/\fB}\to\cO_{\fX^E}\otimes k(\xi) \; \mbox{is trivial}\}.
\]
Let $\fZ_{g,n,\beta'}$ be the moduli stack of genus $g$, $n$-pointed, degree $\beta'$ prestable maps to $Z =Z(s)$; it is an open and closed substack of
$$
\fZ_{g,n} =\Gamma((\fC_{\fM}\times Z)/\fC_{\fM}).
$$
Define
$$
\fZ =\bigcup_{\substack{ \beta'\in H_2(Z;\bZ)\\ i_*\beta'=\beta} } \fZ_{g,n,\beta'}.
$$
\begin{proposition} \label{prop: degeneracy}The degeneracy locus of the cosection is 
$D(\sigma_{\fX^E/\fB})\cong\fZ$.
\end{proposition}
\begin{proof}
Let  $\xi=[(C,x_1,\dots,x_n),f,p]\in\fX^E$. We show $ \langle f^*\delta s(\dot{z}), p\rangle +\langle f^*s,\dot{p}\rangle=0$ for all $\dot{p},\dot{z}$ if and only if $f^*s=0$ and $p=0$. The if direction is clear. Conversely, if $\xi\in D(\sigma_1)$, $f^*s=0$ follows from picking $\dot{z}=0$, $\dot{p}\neq 0$. Then $f:C\to X$ factors as 
\[
C\xrightarrow{g} Z\xhookrightarrow{i} X.
\]
Suppose $p\neq 0$. Then, $H^1(C,f^*E)\neq0$. We just need to show that $H^1(f^*\delta s)$ is surjective, then by Serre's duality we can always pick $\dot{z}$ such that $\langle f^*\delta s(\dot{z}), p\rangle\neq 0$ and $\dot{p}=0$. But we observed that since $s$ is a regular section, $\delta s$ is surjective. So $H^1(f^*\delta s)$ is surjective, as we can see from the long exact sequence induced by $0\to\ker(f^*\delta s)\to f^*A_E\to f^*E\to 0$.
\end{proof}
In other words, $\fZ = \fX^E \times_{\fP} \fB$, where $\fB\to \fP$ is inclusion of the zero section.


\subsection{The cosection factors}\label{ss:lift}

We wish to use this cosection to define a cosection-localized virtual class, following \cite{KL}. To apply the construction, we need to show that the relative cosection $\sigma_{\fX^E/\fB}$ factors through the morphism $\Ob_{\fX^E/\fB}\to\Ob_{\fX^E}$. The absolute obstruction sheaf is defined as follows. Consider 
\[
\nu: \pi_{\fX^E/\fB}^*\bT_{\fB}[-1]\to\bT_{\fX^E/\fB}\xrightarrow{\phi_{\fX^E/\fB}^\vee}\bE_{\fX^E/\fB}^{\vee}
\]
then $\Ob_{\fX^E}=\mbox{Coker}( h^1(\nu))$. Following \cref{rklift}, it suffices to show that 
\begin{equation}\label{vanishlift}
h^1( \sigma^{\bullet}_{\fX^E/\fB}\circ \phi_{\fX^E/\fB})=0
\end{equation}
This is a standard argument, which originally appears in \cite[Section 3.4]{CL}.

Let $\ffH=\pi_{\fB*}\omega_{\pi_{\fB}}$ the Hodge bundle over $\fB$. The morphism $h_{\fX^E/\fB}$ from which we derive the cosection induces $\fX^E\to\ffH$ and a commutative diagram
\begin{equation}\label{coslift}
\begin{tikzcd}
\bT_{\fX^E/\fB} \arrow[r] \arrow[d, "\phi_{\fX^E/\fB}^{\vee}"]    & \pi^*_{\fX^E/\ffH} \bT_{\ffH/\fB} \arrow[d]                        \\
\bE^{\vee}_{\fX^E/\fB} \arrow[r, "\sigma^{\bullet}_{\fX^E/\fB}"] &  \pi^*_{\fX^E/\ffH}R^{\bullet}\pi_{\ffH *}\omega_{\fC_{\ffH}/\ffH}
\end{tikzcd}
\end{equation}
Then \cref{vanishlift} follows from observing $\bT_{\ffH/\fB}$ is concentrated in degree 0 since $\ffH$ is a vector bundle over $\fB$ so in particular $\ffH\to\fB$ is smooth and representable.

\subsection{Virtual classes on $\fZ$.}
To summarize, we now have two constructions for the virtual class of $\fZ$. On one hand, the morphism $\pi_{\fZ/\fB}:\fZ\to\fB$ given by fixing the bundle $N:=i^*E$ on $Z$ gives a perfect obstruction theory $\phi_{\fZ/\fB}:\bE_{\fZ/\fB}\to\bL_{\fZ/\fB}$ where
\begin{equation}\label{potZ}
(\bE_{\fZ/\fB})^{\vee}=R^{\bullet}{\pi_{\fZ}}_*ev_{\fZ}^*(A_{N})
\end{equation}
and $\phi_{\fZ/\fB}$ is determined by the universal evaluation morphism. 
For $A_{N}$ the Atiyah extension
\[
0\to\End(N)\to A_{N}\to T_Z\to 0.
\]
\begin{remark}\label{remZ}
The virtual class $[\fZ/\fB]^{\vir}$ agrees with the usual virtual class $[\fZ/\fM]^{\vir}$. The details are the same as in \cref{potX} and \cref{potXcompare}.
\end{remark}
On the other hand, realizing $\fZ$ as the critical locus of the cosection $\sigma_{\fX^E/\fB}$, gives a cosection localized vitual class $[\fX^E/\fB]_{\sigma_{\fX^E/\fB}}^{\vir}\in A_*(\fZ)$ by \cref{thm:coslocvir}.

 We will use degeneration to the normal cone to show that the two classes agree up to a sign.
\begin{example}[Convex case]
Fix $ g=0$ and $(n,\beta)$, $X$ and $E$ such that $H^1(C,f^*E)=0$ for any stable map $f:C\to X$. This is the case, for example, for $X=\bbP^m$, $E=\ccO(k)$ with $k>0$. Then all the $p$-fields are trivial, and the moduli space $\fX^E$ is just $\fX$. However, the perfect obstruction theories $\bbE_{\fX^E/\fB}$ and $\bbE_{\fX/\fB}$ of these two stacks are different. Together with 
\[
\bbE_{\fX^E/\fX}=\left(R^{\bullet}\pi_{\fX^E *}ev_{\fX^E}^*E^{\vee}\otimes\omega_{\pi_{\fX^E}}\right)^{\vee}
\]
they form a compatible triple in the sense of \cite[Definition 4.5]{VPB}.
Since $R^0\pi_{\fX^E *} ev_{\fX^E}^*E^{\vee}\otimes\omega_{\pi_{\fX}}=0$, we just have a locally-free obstruction sheaf given by $R^1\pi_{\fX^E *} ev_{\fX^E}^*E^{\vee}\otimes\omega_{\pi_{\fX}}$. The virtual fundamental class of $\fX^E$ given by this relative perfect obstruction theory is just 
\begin{align*}
[\fX^E/\fX]^{\vir}&=[\fX/\fB]^{\vir}\cap e(R^1\pi_{\fX^E *} ev_{\fX^E}^*E^{\vee}\otimes\omega_{\pi_{\fX}})\\
&=[\fX/\fB]^{\vir}\cap (-1)^{r(1-g)+\int_{\beta}c_1(E)}e(R^0\pi_{\fX^E *} ev_{\fX^E}^*E) 
\end{align*}
 
The cosection $\sigma_{\fX^E/\fB}$ is compatible with a relative cosection $\sigma_{\fX^E/\fX}$ now given locally by 
\begin{align*}
&H^1(C,f^*E^{\vee}\otimes\omega_C)\to\bbC\\
&\dot{p}\mapsto\langle f^*s,\dot{p} \rangle . 
\end{align*}
Cosection-localization gives a virtual class $[\fX^E]^{\vir}_{\sigma}\in A_*(\fZ)$ and by \cref{thm:coslocvir} we have
\[
\iota_*[\fX^E]^{\vir}_{\sigma}=[\fX^E/\fX]^{\vir}.
\]
In this case, the comparison theorem we aim to prove between $[\fX^E]^{\vir}_{\sigma}$ and $[\fZ]^{\vir}$ is equivalent to the Quantum Lefschetz hyperplane principle of \cite{KKP}, stating that
\[
[\fZ]^{\vir}=[\fX]^{\vir}\cap e(R^0\pi_{\fX} ev_{\fX}^*E).
\]
\end{example}

\section{Degeneration to the normal cone}\label{sec:degeneration}
 Let \[
V=Bl_{Z\times \{0\}}(X\times\bA^1)\setminus\widetilde{X\times\{0\}}.
\]
Let $\pi_{V/\bA^1}$ and $\pi_{V/X}$ denote the projections to $\bA^1$ and to $X$ respectively. Away from the zero fiber, $\pi_{V/X}$ is the projection to the first component, whereas 

\[
{\pi_{V/X}|}_{\pi_{V/\bA^1}^{-1}(0)}: \pi_{V/\bA^1}^{-1}(0)=Vb(N_{Z|X})\to X
\] 
is the projection $\Vb(N)\to Z$ followed by $Z\xrightarrow i X$.
There is a closed embedding $j:V\xhookrightarrow{} \Vb(E)\times\bA^1$ given by $V=Z(\tilde{s})$ for
\begin{equation}\label{tildes}
\tilde{s}=\pi^*s-yt\, ,\;\;\; \tilde{s}\in H^0\left(\Vb(E)\times\bA^1, \pi_{\Vb(E)\times\bA^1/X}^* E\right).
\end{equation}
Here $t$ is the coordinate on $\bA^1$ and $y$ is the section of $\pi^*_{\Vb(E)/X}E$ induced by the identity of $E$.
We can fix a vector bundle $\widetilde{E}:= \pi_{V/X}^*E$ on $V$ and the section $y$, which is $y=t^{-1}s$ away from the zero fiber. Over $t\neq 0$, the triple $(V,\tE,y)$ restricts to $(X,E,t^{-1}s)$ and over $0$ this gives the triple $(\Vb(N),\pi_{\Vb(N)/Z}^*N,1)$ where 1 is the section induced by the identity.

\subsection{The triple $(V,\tE,y)$}
The moduli of sections $\Gamma(\fC_{\fM}\times V/\fC_{\fM})$ is the moduli space of genus $g$, $n$-pointed prestable maps to the smooth variety $V$, denoted $\fV_{g,n}$. Define the degree of a map $f: C\to V$ as the degree of the composite map $\pi_{V/X}\circ f$. Let $\fV$ denote the stack of stable maps to $V$ of degree $\beta$. 
Since the domain curves are projective, stable maps to $V$ factor through the fibers of $\pi_{V/\bA^1}$. We have $\pi_{\fV/\bA^1}:\fV\to \bA^1$.
Similar to \cref{qXSBP}, we have a diagram
\begin{equation}\label{VPB}
\xymatrix{
\fV^{\tE}:=\Spec(\Sym (R^1ev_{\fV}^*\tE) )\ar[r] \ar[d] & \fP\times\bA^1\ar[d]\ar[r]&\fP\ar[d] \\
\fV  \ar[r] \ar[dr] & \fS\times\bA^1  \ar[d]\ar[r] & \fS\ar[d]\\
& \fB\times\bA^1\ar[d] \ar[r] & \fB\ar[d]\\
& \fM\times\bA^1\ar[r] & \fM
}
\end{equation}
Working relative to the last column of diagram (\ref{VPB}), the triple $(V,\tE,y)$  allows to define a (dual) perfect obstruction theory for $\fV$ given by 
\[
\bE_{\fV/\fB}^\vee=R^\bullet\pi_{\fV *}ev_{\fV}^*A_{\tE}.
\]
We have a (dual) perfect obstruction theory for $\fV_{\tE}\to\fB$ given by 
\[
\bE_{\fV^{\tE}/\fB}^\vee=\pi_{\fV^{\tE}/\fV}^*\bE_{\fV/\fB}^\vee\oplus R^{\bullet}\pi_{\fV^{\tE}*}\left(\ev_{\fV^{\tE}}^*\tE^\vee\otimes\omega_{\pi_{\fV^{\tE}}}\right)
\]
and a cosection $\sigma_{\fV^{\tE}/\fB}$ induced by $\delta y:A_{\tE}\to\tE$.

As in \cref{ss: cosection}, we can use the section $y$ to induce a morphism \[
\mathfrak{y}:\fC_{\fB}\times_{BGL_r}V\to\Vb(\cE_{\fB}).\]
 Together with dual pairing on the fibers, $\mathfrak{y}$ induces a morphism 
 \[
 h_{\fV^{\tE}/\fB}:\fC_{\fB}\times_{BGL_r}V\times_{\fC_{\fB}}\Vb(\cE_{\fB}^\vee\otimes\omega_{\pi_{\fB}})\to\ffH
 \]
  from which the cosection derives. This construction guarantees at once that the cosection lifts, as observed in \cref{ss:lift}.
  
\begin{lemma}\label{deglocy}
The degeneracy locus of the cosection is
\[D(\sigma_{\fV^{\tE}/\fB})=\fZ\times\bA^1\subset\fV^{\tE}.
\] 
\end{lemma} 
\begin{proof}
Then the situation is analogous to \cref{prop: degeneracy}, because $y$ is a regular section of $\tE$ over $V$ with vanishing locus $\fZ\times\bA^1$. Indeed, $\{y=0\}\subset V$ is $\{\pi^*s-ty=0, y=0\}\subset \Vb(E)\times\bA^1$ which is $Z\times\bA^1$. To check that the section $y$ is regular, we check that the inclusion of the vanishing locus $Z\times\bA^1\subset V$ is a regular embedding. Fibers $V_t$ are effective Cartier divisors, so if $Z\subset V_t$ is regular for each $t$, then $Z\times\bA^1\subset V$ is regular. But on fibers the embedding is given by the vanishing of a regular section of a vector bundle, so it is regular. 

Then by the analog of \cref{prop: degeneracy}  we have that $D(\sigma_{\fV^{\tE}/\fB})$ is the moduli of stable maps to $Z\times\bA^1$, which we can identify with $\fZ\times\bA^1$ since the image of a projective curve on $\bA^1$ needs to be constant.
\end{proof}

This is a good place to point out that the formation of the Atiyah extension does not commute with pullback, in particular the restriction of $A_{\tE}$ to $V|_t\cong X$ is not isomorphic to the bundle $A_E$ since $\mbox{rk}(A_{\tE})=\mbox{rk}(A_E)+1$. 
We need to consider the moduli spaces $\fV$ and $\fV^{\tE}$ relative to the penultimate column of (\ref{VPB}). To construct a perfect obstruction theory for $\fV\to\fB\times\bA^1$, consider the morphism
\[
V\to BGL_r\times\bA^1
\]
induced by $\pi_{V/\bA^1}$ and the bundle $\tE$. As we check below, this morphism is smooth and induces a  relative Atiyah extension $A$ and a short exact sequence coming from the distinguished triangle $V\to  BGL_r\times\bA^1\to BGL_r$
\begin{equation}\label{SESA}
0\to A\to A_{\tE}\to\cO_V\to 0.
\end{equation}

\begin{lemma}\label{VrelA}
The morphism $f:=(\psi_{\tE},\pi_{V/\bA^1}):V\to BGL_r\times\bA^1$ is smooth. It has a locally free tangent sheaf $A$ on $V$ or rank $r^2+n$. Over $t\neq 0$, we can identify $V_t\cong X$ and $A$ restricts to the sheaf $A_E$. Over $V_0\cong \Vb(N)$, $A$ restricts to $\pi_{\Vb(N)/Z}^*(A_N\oplus N)$. 
\end{lemma}
\begin{proof}
To check smoothness consider the smooth atlas $\bA^1\to  BGL_r\times\bA^1$, we have the following Cartesian diagram
\[
\xymatrix{
P_{\tE}\ar[r]^{\tilde{f}}\ar[d] & \bA^1\ar[d]\\
V\ar[r]^{f} &  BGL_r\times\bA^1
}
\]
where  $P_{\tE}$ is the $GL_r$-torsor on $V$ associated to the bundle $\tE$. We check smoothness by giving an explicit expression for the top morphism $\tilde{f}:P_{\tE}\to\bA^1$. To this end, we describe below an embedding of $P_{\tE}$ into the torsor corresponding to the vector bundle $\pi_{\Vb(E)\times\bA^1/X}^* E$ over $\Vb(E)\times\bA^1$. The composition of this embedding with the projection to the $\bA^1$ factor coincides with the map $\tilde{f}$. 

Since $E$ trivializes when pulled back by $P_E\to X$, we have the following cartesian diagram of $GL_r$-torsors
\begin{equation*}
\xymatrix{
P_{\tE}\ar[r] \ar[d]& P_E\times\bC^r\times\bA^1\ar[r] \ar[d]& P_E\ar[d]\\
V\ar[r]^{j\qquad} & \Vb(E)\times\bA^1\ar[r] & X.
}
\end{equation*}
The right $GL_r$-action on $P_E\times\bC^r\times\bA^1$ is given by
\[
(e,v,t)g=(eg,\rho_0(g^{-1})v,t).
\]
The section $\tilde{s}$ of \cref{tildes} can be lifted to a $GL_r$-equivariant map
\begin{align}\label{PCA}
\phi_{\tilde{s}}:&P_E\times\bC^r\times\bA^1\lra {\bC}^r\nonumber\\
&(e,v,t)\mapsto \phi_{s}(e)+tv
\end{align}
where $GL_r$ acts on the right on $\bC^r$ by $vg=\rho_0(g^{-1})v$.
The torsor $P_{\tE}$ is vanishing locus of this regular function $P_{\tE}\cong \phi_{\tilde{s}}^{-1}(0)$. To show $\tilde{f}$ is smooth by \cite[Theorem 10.2]{Ha} it suffices to show that the geometric fibers are regular and equidimensional of dimension $\dim{P_{\tE}}-1=r^2+n$. By \cref{PCA} we see that the fibers over $t\neq 0$ consist of the graph of $\phi_s$, so are isomorphic to $P_E$ and the fiber over zero to $\phi_s^{-1}(0)\times\bC^r\cong P_N\times \bC^r$. So all geometric fibers are smooth of the correct dimension. 
We can then define $A=[T_{P_{\tE}/\bA^1}/GL_r]\cong T_{V/BGL_r\times\bA^1}$. Let $i_t:t\xhookrightarrow{} \bA^1$ be the inclusion of a point, the restriction of $f$ to the fiber over $t$ is the morphism 
\[
\psi_{\tE|_{V_t}}:V_t\to BGL_r
\]
induced by the restriction of $\tE$. So $i_t^*A\cong T_{V_t/BGL_r}$, which is the bundle defined by the Atiyah extension of $i_t^*\tE$. For $t\neq 0$, after identifying $V_t\cong X$ we have $i_t^*\tE\cong E$ and $i_t^*A\cong A_E$. Over $t=0$, $V_0\cong\Vb(N)$ and $i_0^* \tE\cong \pi_{\Vb(N)/Z}^*N$. The corresponding $GL_r$-torsor is $P_N\times\bC^r\to\Vb(N)$. Then $i_0^*A\cong\pi_{\Vb(N)/Z}^*( A_N\oplus N)$. 
\end{proof}
Let $\fC_{\fB\times\bA^1}:=\fC_{\fB}\times\bA^1$, we can identify $\fV$ with an open substack of the moduli of sections of $V\times_{BGL_r\times \bA^1}\fC_{\fB\times\bA^1}$ over $\fC_{\fB\times\bA^1}$, then by the usual moduli of sections construction the bundle $A$ induces a (dual) perfect obstruction theory 
\[
\bE^\vee_{\fV/\fB\times\bA^1}=R^\bullet\pi_{\fV *} ev_{\fV}^* A.
\]
Similarly, for $\fV^{\tilde{E}}$, which is (an open substack of) the moduli of sections of $(V\times_{BGL_r\times \bA^1}\fC_{\fB\times\bA^1})\times_{\fC_{\fB\times\bA^1}}\Vb(\fE_{\fB\times\bA^1}^{\vee}\otimes\omega_{\pi_{\fB\times\bA^1}})$ we obtain
\[
\bE^\vee_{\fV^{\tE}/\fB\times\bA^1}=R^\bullet\pi_{\fV^{\tE} *} \left(ev_{\fV^{\tE}}^* A\oplus ev_{\fV^{\tE}}^*\tE^{\vee}\otimes\omega_{\pi_{\fV^{\tE}}}\right).
\]
 There is a cosection $\sigma_{\fV^{\tE}/\fB\times\bA^1}$ induced by $A\to A_{\tE}\xrightarrow{\delta y} \tE$ together with the usual dual pairing. At the level of moduli of sections, it is induced by a map to $\Vb(\omega_{\pi_{\fB\times\bA^1}})$, which guarantees that it lifts by the argument of \cref{ss:lift}.
 The degereracy locus is the same as that of $\sigma_{\fV^{\tE}/\fB}$, because the above composition is surjective.
 We have a cosection localized virtual fundamental class
 \[
[\fV^{\tE}/\fB\times\bA^1]^{\vir}_{\sigma_{\fV^{\tE}/\fB\times\bA^1}}\in A_*(\fZ\times\bA^1).
 \]
 However, $\bE_{\fV/\fB\times\bA^1}$ and $\bE_{\fV/\fB}$ are not compatible perfect obstruction theories in the sense of  \cref{a1inv} , and do not give the same absolute obstruction sheaf. This is because its obstruction bundle $\Ob_{\fV^{\tilde{E}}/\fB}$ has an extra factor of $R^1\pi_{\fV^{\tilde{E} }*}\ccO_{\fC_{\fV^{\tilde{E}}}}$, as we can see from \cref{SESA}. To remedy this, we need to introduce a different obstruction theory $\bE'_{\fV^{\tilde{E}}/\fB}$. 
We can argue as in \cite[Theorem 4.6]{chang2020invariants}, by defining $\bE'^{\vee}_{\fV^{\tilde{E}}/\fB}$ as the mapping cone of $f:\ccO_{\fV^{\tilde{E}}}[-1]\to\bE^{\vee}_{\fV^{\tE}/\fB\times\bA^1}$. We have 
 \[\xymatrix{
 \ccO_{\fV^{\tilde{E}}}[-1]\ar[r]^{f}& \bE^{\vee}_{\fV^{\tilde{E}}/\bA^1\times \fB} \ar[r] & \bE'^{\vee}_{\fV^{\tilde{E}}/\fB}\ar[r]^{+1} & \\
\ccO_{\fV^{\tilde{E}}}[-1]\ar[r]\ar[u]^{\cong}& \bT_{\fV^{\tilde{E}}\bA^1\times \fB} \ar[u]^{\phi^{\vee}_{\fV^{\tilde{E}}\bA^1\times \fB}}\ar[r] &\bT_{\fV^{\tilde{E}}/\fB}\ar[u]^{\phi^{\vee}_{\fV^{\tilde{E}}/\fB}}\ar[r]^{+1} &
 }
 \]
where $\phi^{\vee}_{\fV^{\tilde{E}}\bA^1\times \fB}$ is induced by the mapping cone axiom and $\bE'^{\vee}_{\fV^{\tilde{E}}/\fB}$ defines a (dual) perfect obstruction theory as we can easily check from the long exact sequences coming from the rows of the above diagram. The cosection $\sigma_{\fV^{\tilde{E}}/\fB}$ costructed above lifts to the absolute obstruction sheaf, so it induces a cosection $\sigma'_{\fV^{\tilde{E}}/\fB}$ of $\Ob'_{\fV^{\tilde{E}}/\fB}=h^1(\bE'^{\vee}_{\fV^{\tilde{E}}/\fB})$. 
This gives us the required perfect obstruction theories on $\fV^{\tilde{E}}$ to apply \cref{deformationthm}.
 
\subsection{The fiber over $0$}\label{ss:zerofiber} The triple  $(\Vb(N),\pi_{\Vb(N)/Z}^*N,1)$ gives rise to a moduli space $\fN^{N}$ of stable maps to $\Vb(N)$ with p-fields together with a dual relative perfect obstruction theory
\begin{equation}\label{potN}
\bE^\vee_{\fN^N/\fB}=R^\bullet\pi_{\fN^N/\fB *}ev_{\fN^N}^*\pi_{\Vb(N)/Z}^*( A_N\oplus N)\oplus ev_{\fN^N}^*\pi_{\Vb(N)/Z}^*N^\vee\otimes\omega_{\pi_{\fN^N}}
\end{equation}
and a cosection $\sigma_{\fN^N/\fB}$ induced by $\delta (1)$. By \cref{VrelA}, and base-change, we see that $\bE_{\fN^N/\fB}=\iota_0^*\bE_{\fV^{\tilde{E}}/\fB\times\bA^1}$. We note here that the degree of a map to $\Vb(N)$ was defined to be $\pi_{\Vb(N)/Z*}i_*[C]\in H_2(X;\bbZ)$ so we do obtain that $D(\sigma_{\fN^N/\fB})$ is the stack of stable maps to $Z$ of degrees $\beta'$ such that $i_*\beta'=\beta$.

To describe this cosection explicitly, recall that the $GL_r$-torsor associated to \[\pi_{\Vb(N)/Z}^*N\to\Vb(N)\] is $P_N\times\bC^r$ and the tautological section $1$ of the vector bundle induces the  $GL_r$-equivariant map $P_N\times\bC^r\to\bC^r$ given by projection onto the second coordinate. Then $\delta (1)$ is the projection $ A_N\oplus N\to N$. The section $1$ of $\pi_{\Vb(N)/Z}^*N\to\Vb(N)$ is regular and has vanishing locus $Z$  regularly embedded in $\Vb(N)$ by the zero section. So by the argument in \cref{prop: degeneracy} have that the degeneracy locus of the cosection is $\fZ$. We obtain the following: 
 
 \begin{proposition}
The triple $(\Vb(N),\pi_{\Vb(N)/Z}^*N,1)$ defines the moduli space of stable maps to $\Vb(N)$ with p-fields, with a perfect obstruction theory and a cosection localized virtual class
\[
[\fN^N/\fB]_{\sigma_{\fN^N/\fB}}^\vir\in A_*(\fZ).
\]
\end{proposition}
We can also view $\fN^N$ as a cone over $\fZ$. In fact composing a stable map to $\Vb(N)$ with the projection to $Z$ gives a morphism $\fN^N\to\fZ$. 
Let $\cE_{\fZ}=ev_{\fZ}^*N$, by cohomology and base-change we can rewrite the dual perfect obstruction theory of $\fN^N/\fB$ as
\begin{equation}
\bE_{\fN^N/\fB}^{\vee}=\pi^*_{\fN^N/\fZ}\left(\bE_{\fZ/\fB}^{\vee}\oplus R^{\bullet}\pi_{\fZ*}(\cE_{\fZ}\oplus\cE_{\fZ}^{\vee}\otimes\omega_{\fZ})\right).
\end{equation}
If we let $\cE_{\fN^N}=\pi_{\fN^N/\fZ}^*\cE_{\fZ}$, note that this is also isomorphic to $ev_{\fN^N}^*\pi_{\Vb(N)/Z}^*N$, the relative obstruction sheaf becomes 
\[
\Ob_{\fN^N/\fB}=\pi_{\fN^N/\fZ}^*\Ob_{\fZ/\fB}\oplus R^1\pi_{\fN^N *}\left( \cE_{\fN^N}\oplus\cE_{\fN^N}^{\vee}\otimes\omega_{\pi_{\fN^N}}\right)
\]
and the cosection $\sigma_{\fN^N/\fB}$ simply comes from the natural pairing of $\cE_{\fZ}\oplus\cE^{\vee}_{\fZ}\to\cO_{\fZ}$.

\begin{proposition}\label{prop:potN^N}
We can identify
\[
\fN^N=\Spec_{\fZ}\Sym(R^1(\cE_{\fZ})\oplus R^1(\cE_{\fZ}^{\vee}\otimes\omega_{\pi_{\fZ}})).
\]
That is, $\fN^N\cong \fZ\times_{\fB}\fP$ where the morphism $\fZ\to \fB$ is induced by the bundle $N$ over $Z$. This induces a (dual) perfect obstruction theory of virtual dimension 0
\[
\bE^\vee_{\fN^N/\fZ}=\pi^*_{\fN^N/\fZ}R^\bullet\pi_{\fZ *}\cE_{\fZ}\oplus \cE_{\fZ}^\vee\otimes\omega_{\pi_{\fZ}}
\]
which is compatible with $\bE^\vee_{\fN^N/\fB}$ from \cref{potN} and $\bE^\vee_{\fZ/\fB}$ from \cref{potZ}.
\end{proposition}
\begin{proof}
The first claim follows by observing that the moduli of stable maps to $\Vb(N)$ is the cone $\Spec_{\fZ}\Sym(R^1(ev_{\fZ}^*N))$ over $\fZ$. In fact, a map $f:C\to\Vb(N)$ is given by the composite map $g:C\to\Vb(N)\to Z$ together with a map of $\cO_Z$-algebras $\Sym^*(N^\vee)\to f_*\cO_C$. This is determined by the image of degree 1 generators. Hence, by adjunction, corresponds to a map of $\cO_C$-modules $g^*N^\vee\to\cO_C$, which gives some $q\in H^0(C,g^*N)$. In the same way, $(g,q)$ determine $f:C\to\Vb(N)$ wich commutes with $g$ and $\pi_{\Vb(N)/Z}$. 
Given a triple $((C,x_1,\dots ,x_n),f,p)$ in $\fN^N(\bullet)$, where $p\in H^0(C,f^*\pi^*_{\Vb(N)/Z}N^\vee\otimes\omega_C)$, we can construct uniquely a quadruple $((C,x_1,\dots ,x_n),g, q,p)$ where $(g,q)$ are as described above and $p\in H^0(C,g^*N^\vee\otimes\omega_C)$ since $g=\pi_{\Vb(N)/Z}\circ f$. Conversely, $(g,q)$ determine a morphism $f: C\to \Vb(N)$ over $g:C\to Z$.

Then the claim that $\bE_{\fN^N/\fZ}^\vee$ is a perfect obstruction theory follows from the result on perfect obstruction theories of moduli of sections \cite[Proposition 2.5]{CL} and cohomology and base change, see for example \cite[Theorem 5.2]{VPB}. 
Similarly, the claim about compatibility follows from pulling back the split exact sequence
\[
0\to ev_{\fZ}^*N\oplus ev_{\fZ}^*N^\vee\otimes \omega_{\pi_{\fZ}}\to ev_{\fZ}^*N\oplus ev_{\fZ}^*N^\vee\otimes\omega_{ \pi_{\fZ}}\oplus ev_{\fZ}^*A_N\to ev_{\fZ}^*A_N\to 0
\]
to $\fC_{\fN^N}$ and using the compatibility of evaluation morphisms. 
\end{proof}
 Using this, we can think of objects in $\fN^N(\bullet)$ as quadruples $((C,x_1,\dots ,x_n),f,p,q)$, where $((C,x_1,\dots ,x_n),f)$ is a stable map to $Z$ of degree $\beta'$ for some $i_*\beta '=\beta$, $q\in H^0(C,f^*N)$ and $p\in H^0(C,f^*N^\vee\otimes\omega_C)$. The cosection $\sigma_{\fN^N/\fB}$ agrees with that induced by $\sigma_{\fP/\fB}$. It has local expression
\begin{equation}\label{cosectionNN}
\sigma_{\fN^N/\fB}|_{\xi}: H^1(C,f^*A_N)\oplus H^1(C,f^*N)\oplus H^1(C,f^*N^\vee\otimes\omega_C)\to\bC
\end{equation}
given by $(\dot{f}, \dot{q},\dot{p})\mapsto \langle \dot{p},q \rangle+\langle p,\dot{q}\rangle$ over $\xi=[(C,x_1,\dots ,x_n),f,p,q]$.

\subsection{Deformation invariance}\label{ss:deform}
We are now in the position to apply \cref{deformationthm}, since we have $\fV^{\tilde{E}}$ with compatible perfect obstruction theories, relative to $\fB$ and $\fB\times\bA^1$ respectively, given by $\bE'_{\fV^{\tE}/\fB}$ and $\bE_{\fV^{\tE}/\fB\times\bA^1}$ . 
From \cref{VrelA} we can see that 
\[
\bE_{\fX^E/\fB}=\iota_1^*\bE_{\fV^{\tE}/\fB\times\bA^1}
\]
 since the restriction of $A$ to $V_1$ is the bundle $A_E$. So we just need to check that the bundle map $A\to A_{\tE}\to\tE$ restricted to $V_1$ agrees with $\delta s:A_E\to E$. This guarantees that the cosection $\sigma_{\fX^E/\fB}$ of \cref{ss: cosection} agrees with the one induced by $\sigma_{\fV^{\tE}/\fB}$. But $A\to\tE$ is induced by the differential of a morphism $V\to U_r\times \bA^1$ over $BGL_r\times\bA^1$, where $V\to U_r$ is induced by the section $y$. Restricting to $X\cong V_1$ this is $\delta s:A_E\to E$. 
Similarly,
\[
\bE_{\fN^N/\fB}=\iota_0^*\bE_{\fV^{\tE}/\fB\times\bA^1}
\]
and $\sigma_{\fN^N/\fB}$ is induced by restriction from $\sigma_{\fV^{\tE}/\fB}$.
\begin{theorem}\label{thm:def}
\[
[\fN^N/\fB]^\vir_{\sigma_{\fN^N/\fB}}=[\fX^E/\fB]^\vir_{\sigma_{\fX^E/\fB}}\in A_*(\fZ)
\]
\end{theorem}
\begin{proof}
By \cref{deformationthm} we have that $[\fX^E/\fB]^\vir_{\sigma_{\fX^E/\fB}}=i_1^![\fV^{\tE}/\fB]_{\sigma_{\fV^{\tE}}/\fB}$ and $[\fN^N/\fB]^\vir_{\sigma_{\fN^N/\fB}}=i_0^![\fV^{\tE}/\fB]_{\sigma_{\fV^{\tE}}/\fB}$. So the two classes are equal since $i_0^!=i_1^!:A_*(\fZ\times\bA^1)\to A_*(\fZ)$.  
\end{proof}

\section{Torus localization}\label{section:localization}
To compare $[\fN^N/\fZ]^\vir_{\sigma_{\fN^N/\fZ}}$ with $[\fZ/\fB]^\vir$, we will use virtual torus localization for cosection-localized virtual classes \cref{cosloc}.  

Let $T=\bC^*$ act on $Z$ trivially, and choose a non-trivial equivariant structure on $N$. We may choose the action that scales the fibers of $N$ with weight 1 and the fibers of the dual $N^{\vee}$ with weight -1. This induces a $T$-action on $\fN^N$, on an object $\xi$ in $\fN^N(\bullet)$, $t\in T$ acts by
\[
t((C,x),f,q,p,)=((C,x),f,tq,t^{-1}p)
\]
The fixed locus is $(\fN^N)^T\cong\fZ\xrightarrow{\iota}\fN^N$, the inclusion is by vanishing of the $p$ and $q$-fields and geometrically is just the inclusion of the vertex of a cone, and $\iota\circ\pi_{\fN^N/\fZ}=Id_{\fZ}$. By \cref{prop:potN^N}, we can identify $\fN^N$ with an open substack of the moduli of sections of 
\begin{equation}\label{evalN^N}
\cZ_{\fN^N}:=\left(Z\times_{BLG_r}\fC_{\fB}\right)\times_{\fC_{\fB}}\Vb\left(\cE_{\fB}\oplus \cE_{\fB}^{\vee}\otimes\omega_{\fB}\right)\to
\fC_{\fB}
\end{equation}
over $\fB$. Let $T$ act with weight 1 on $\cE_{\fB}$ and weight -1 on $\cE_{\fB}$, trivially on $\fB$ and $\fC_{\fB}$. The morphism $\fN^N\to\fB$ as well as the universal evaluation $Ev_{\fN^N}:\fC_{\fN^N}\to \left(Z\times_{BLG_r}\fC_{\fB}\right)\times_{\fC_{\fB}}\Vb\left(\cE_{\fB}\oplus \cE_{\fB}^{\vee}\otimes\omega_{\fB}\right)$ are $T$-equivariant.

\begin{lemma}
The relative perfect obstruction theory $\bE_{\fN^N/\fB}$ is $T$-equivariant. Restricted to the $T$-fixed locus $\fZ$ it splits as $\bE_{\fN^N/\fB}|_{\fZ}=\bE_{\fN^N/\fB}|_{\fZ}^\fixd\oplus\bE_{\fN^N/\fB}|_{\fZ}^\movn$. The fixed part is the usual perfect obstruction theory for $\fZ/\fB$ of \cref{potZ}. 
The moving part is $N^\vir=(\bE_{\fN^N/\fB}|_{\fZ}^{\movn})^\vee=R^\bullet\pi_{\fZ *}(\cE_{\fZ}\oplus\cE_{\fZ}^{\vee}\otimes\omega_{\fZ})$.
\end{lemma}
\begin{proof}
The perfect obstruction theory $\bE_{\fN^N/\fB}$ described in \cref{potN} is quasi-isomorphic to the one induced by induced by the universal evaluation of $Ev_{\fN^N}$ of \cref{evalN^N}. This follows from the discussion in \cref{ss:zerofiber}.  The diagram
\[
\xymatrix{
 & \cZ_{\fN^N}\ar[d]\\
 \fC_{\fN^N}\ar[r]\ar[ur]^{Ev_{\fN^N}} & \fC_{\fB}\\
}
\]
is equivariant. Then we have a $T$-equivariant morphism of $T$-equivariant cotangent complexes $\bL_{\cZ_{\fN^N}/\fC_{\fB}}\to Ev_{\fN^N}^*\bL_{\fC_{\fN^N}/\fC_{\fB}}$. Applying $R^{\bullet}\pi_{\fN^N *}$ to this morphism recovers the perfect obstruction theory of $\fN^N\to\fB$ as a $T$-equivariant morphism. The dual perfect obstruction theory of $\fN^N/\fB$ is quasi-isomorphic to 
\[
\bE_{\fN^N/\fB}^{\vee}=\pi^*_{\fN^N/\fZ}\left(\bE_{\fZ/\fB}^{\vee}\oplus R^{\bullet}\pi_{\fZ*}(\cE_{\fZ}\oplus\cE_{\fZ}^{\vee}\otimes\omega_{\fZ})\right).
\]
The torus $T$ acts on the second component only. Recall that $T$ acts trivially on $Z$ and with opposite weights on $N$ and $N^{\vee}$ so the torus action on $A_N$ is trivial by the short exact sequence
\[
0\to N^{\vee}\otimes N\to A_N\to T_Z\to 0.
\]
The restriction to the fixed locus $\fZ\xhookrightarrow{}\fN^N$ is $\bE_{\fN^N/\fB}|_{\fZ}=\bE_{\fZ/\fB}\oplus\pi^*_{\fZ/\fB} 0^*\bE_{\fP/\fB}$ where $0$ denotes the inclusion of $\fB$ in $\fP$ as the zero section. The fixed part is $\bE_{\fN^N/\fB}|_{\fZ}^{\fixd}=\bE_{\fZ/\fB}$. The moving part is $N^\vir=R^\bullet\pi_{\fZ *}(\cE_{\fZ}\oplus\cE_{\fZ}^{\vee}\otimes\omega_{\fZ})$.
\end{proof}

Consider the absolute perfect obstruction theory $\bE_{\fN^N}$ induced by $\bE_{\fN^N/\fB}$. That is the perfect obstruction theory which completes the equivariant commutative diagram
\[
\xymatrix{
\bE_{\fN^N/\fB} \ar[r]\ar[d]&\bL_{\fB}[1]\ar[d]^{=}\\ 
\bL_{\fN^N/\fB}\ar[r] & \bL_{\fB}[1]
}
\]
to an equivariant distinguished triangle in the $T$-equivariant derived category of $\fN^N$
\[
\xymatrix{
\bE_{\fN^N}\ar[d]\ar[r]& \bE_{\fN^N/\fB} \ar[r]\ar[d]&\bL_{\fB}[1]\ar[d]^{=}\\ 
\bL_{\fN^N}\ar[r] & \bL_{\fN^N/\fB}\ar[r] & \bL_{\fB}[1].
}
\]
Since $\fB$ is a smooth Artin stack, taking the long exact sequence in cohomology shows that the complex $\bE_{\fN^N}$ is concentrated in degrees $[-1,0]$. It is a perfect obstruction theory by  the 5-lemma and is $T$-equivariant by construction. 
The obstruction sheaf $\Ob_{\fN^N}=h^1(\bE_{\fN^N}^{\vee})$ is the same as that constructed \'{e}tale locally in \cite[Definition 4.1]{KL}. The cosection $\sigma_{\fN^N/\fB}$ induced by the tautological section of the vector bundle $\pi_{\Vb(N)/Z}^*N$ is $T$-equivariant by the expression in \cref{cosectionNN}. So the absolute cosection $\sigma_{\fN^N}$ is $T$-equivariant.
To apply the cosection localized torus localization result \cite[Theorem 3.5]{CKL}, we need to show that $N^\vir$ has a global equivariant resolution $[N_0\to N_1]$ by locally free sheaves.
Following in \cite[Proposition 5]{GWinAG}, choose some ample line bundle $L$ on $Z$ with trivial $T$-action. Now 
\[
\cL\coloneqq ev_{\fZ}^*L\otimes \omega_{\pi_{\fZ}}(\Sigma)
\]
 is $\pi_{\fZ}$-ample, where $\Sigma$ is union of the image of the marked points $s_i:\fZ\to\fC_{\fZ}$, and has trivial $T$-action. Then for $N$ sufficiently large 
 \[
 \pi_{\fZ}^*\pi_{\fZ*}(\cE_{\fZ}\otimes\cL^{\otimes N})\to(\cE_{\fZ}\otimes\cL^{\otimes N})
 \]
 is surjective and 
 \[
 R^1\pi_{\fZ*}(\cE_{\fZ}\otimes\cL^{\otimes N})=0.
 \]
  Fix this $N$ and let 
  \[
  F_1=\pi_{\fZ}^*\pi_{\fZ*}(\cE_{\fZ}\otimes\cL^{\otimes N})\otimes\cL^{\otimes-N}.
  \]
  So $F_1\to\cE_{\fZ}$ is surjective and $T$-equivariant by construction. The is kernel some locally free sheaf $F_0$. By construction, $R^1\pi_{\fZ*} F_1$ (hence $R^1\pi_{\fZ*} F_0$) is locally free. Then we have
  \[
  0\to \pi_{\fZ*}(\cE_{\fZ})\to R^1\pi_{\fZ*}F_0\to R^1\pi_{\fZ*}F_1\to R^1\pi_{\fZ*}(\cE_{\fZ})\to 0
  \]
 Let $\widetilde{F_i}:=R^1\pi_{\fZ*}F_i$, we can take $[\widetilde{F_0}\oplus \widetilde{F_1}^{\vee}\to \widetilde{F_1}\oplus \widetilde{F_0}^{\vee}]$ as an equivariant locally free resolution for $R^\bullet\pi_{\fZ *}(\cE_{\fZ}\oplus\cE_{\fZ}^{\vee}\otimes\omega_{\fZ})$.
which is the required resolution. We are ready to apply \cref{cosloc}, which in our situation asserts that 
\[
[\fN^N]_{\sigma_{\fN^N}}^\vir=[\fN^N/\fB]_{\sigma_{\fN^N/\fB}}^\vir=\frac{[\fZ/\fB]^\vir}{e_T(N^{\vir})}\in A^T_*(\fZ)\otimes_{\bbQ}\bbQ[t,t^{-1}].
\]

By definition,
\[
e_T(N^{\vir})=\frac{e_T(R^1\pi_{\fZ*}N_0)}{e_T(R^1\pi_{\fZ*}N_1)}.
\]
for a locally free resolution $[N_0\to N_1]\cong N^\vir$. Let $[\widetilde{F_0}\to \widetilde{F_1}]$ be a global locally-free resolution by sheaves of rank $f_0$ and $f_1$ respectively of $R^\bullet\pi_{\fZ*}\cE_{\fZ}$, constructed above. Then 
\[
e_T(N^{\vir})=\frac{e_T(\widetilde{F_0})e_T(\widetilde{F_1}^{\vee})}{e_T(\widetilde{F_1})e_T(\widetilde{F_0}^{\vee})}=(-1)^{f_0+f_1}.
\]
Observe that
\[
f_0+f_1\equiv f_0-f_1=\chi(f^*N)=\int_{\beta}c_1(E)+r(1-g) \;\; \mbox{mod} 2.
\]
By the results of \cref{thm:def}, \cref{remZ} and \cref{cosloc} we conclude the following.
\begin{theorem}
The cosection localized virtual class of the moduli of stable maps with $p$-fields to $X$ agrees up to a sign with the virtual class of the moduli of stable maps to $Z$:
\[
[\fX^E/\fB]^{\vir}_{\sigma_{\fX^E/\fB}}=(-1)^{\int_{\beta}c_1(E)+r(1-g)}[\fZ/\fB]^{vir}=(-1)^{\int_{\beta}c_1(E)+r(1-g)}[\fZ/\fM]^{vir}.
\]
\end{theorem}

\section{The Deligne--Mumford case}

\subsection{Target}
We now adapt our discussion to consider a more general target. Let $\ccX$ denote a smooth Deligne--Mumord stack with projective coarse moduli $X$. Let $\ccE$ be a locally-free sheaf of rank $r$ on $\ccX$ and $s$ denote a regular section. The vanishing locus of $s$ is a smooth Deligne--Mumford stack $\ccZ$ with projective coarse moduli space $Z$. \'{E}tale-locally $\ccX$ is the quotient $[U/\Gamma]$ of an affine scheme $U$ by a finite group $\Gamma$ and $\ccE$ is an equivariant vector bundle $F$ over $U$. The section $s$ locally is some $\Gamma$-equivariant section $\tilde{s}:U\to F$, which gives a local description of $\ccZ$ as $[Z(\tilde{s})/\Gamma]$.

We also introduce the inertia stacks of $\ccX$ and $\ccZ$,  these will be the target of the evaluation morphisms at the marked points. These evaluations did not play a role in the non-orbifold construction, but will be important here to compute the virtual dimensions of the moduli spaces. 

 Let 
\[
I\ccX=\bigsqcup_{j\in\bbN}\Hom\mathrm{Rep}(B\mu_j,\ccX)=\bigsqcup_{\lambda\in\Lambda}\ccX_{\lambda}
\]
be the inertia stack of $\ccX$. An object of $I\ccX$ is a pair $(x,g)$ of an object $x$ of $\ccX$ and $g\in\Aut_{\ccX}(x)$. $\ccX_\lambda$ are the connected components, indexed by the finite set $\Lambda$. On each conencted component, $\mathrm{order}(g)$ is constant. There is a natural projection $q:I\ccX\to\ccX$ which is a finite \'{e}tale morphism.

$\ccZ$ is a closed substack of $\ccX$, in particular $\Aut_{\ccZ}(x)=\Aut_{\ccX}(x)$ for points $x$ of $\ccZ$. Then $I\ccZ$ is a substack of $I\ccX$ and we can decompose the inertia stack of $\ccZ$ as 
\[
I\ccZ=\bigsqcup_{\lambda\in\Lambda}\ccZ_{\lambda}.
\]
The substacks $\ccZ_{\lambda}$ may be disconnected (if $\ccZ$ itself is disconnected).

\subsubsection{The pair $(\ccX,\ccE)$}
We can consider the frame bundle of the vector bundle $\Vb(\ccE)$, which is a morphism
\[
\psi_{\ccE}:\ccX\to BGL_r=[\bullet/GL_r]
\]
as before, this morphism is smooth and defines a locally-free sheaf $A_{\ccE}=h^0(\bbT_{\psi_{\ccE}})$ which fits into a short exact sequence

\begin{equation}\label{AcurlyE}
0\to\End(\ccE)\to A_{\ccE}\to T_{\ccX}\to 0.
\end{equation}

The formation of the Atiyah sequence above is possible for DM stacks by the observation that for $f:U\to V$ an \'{e}tale morphism and $F$ a vector bundle on $V$, we have $f^*A_{F}\cong A_{f^*F}$ as detailed for example in \cite{MR2923408}.

\subsection{A section $s$ of $E$.}
Let $s$ be a section of the vector bundle $\Vb(\ccE)\to\ccX$. This defines a morphism $\psi_s$ as below
\[
\xymatrix{
\Vb(\ccE)\ar[r]\ar[d] & [\bbC^r/GL_r]\ar[d]\\
\ccX\ar[r]^{\psi_{\ccE}}\ar[ur]^{\psi_s} & BGL_r.
}
\]
Define $\delta s$ as the relative differential of $\psi_s$:
\[
\delta s: A_{\ccE}=\bbT_{\ccX/BGL_r}\to\psi_s^*\bbT_{[\bbC^r/GL_r]/BGL_r}=\ccE.
\]

\subsubsection{Regularity} 
We now assume that $s$ is a {regular} section, in the sense that $s_U$ is a regular section of $\ccE_U$ for any \'etale cover $U\to\ccX$.
Then $\ccZ$ is a smooth substack of $\ccX$ of codimension $r$, with projective coarse moduli space $Z$. This also implies that
$$
\delta s: A_{\ccE}\to \ccE
$$  
is a surjective map of locally-free sheaves.  
As before, the section $s$ can be seen as defining a superpotential
\begin{align*}
W:\Vb(\ccE^{\vee})\to\bbC\\
(x,p)\mapsto\langle s(x),p\rangle.
\end{align*}
Since $s$ is regular, the degeneracy locus of $W$ is $\ccZ$. 

\subsection{Quantum Version}
We need to make modifications to account for the fact that the target of our stable maps may have non-trivial stabilizers. This requires to extend the domain curves being considered to allow for some orbifold structure at nodes and marked points.
\subsection{Moduli stacks independent of $(\ccX,\ccE,s)$}\label{ss: SandP}
Let $\fM^{\tw}=\fM^{\tw}_{g,n}$  be the moduli stack of genus $g$, $n$-pointed twisted balanced curves with sections of all the gerbes \cite[Definition 3.5.5]{AGV01}. For a connected scheme $T$, an object of $\fM^{\tw}$ is a diagram
\[
\xymatrix{
\Sigma_i\ar[dr]\subset&C_T\ar[d]\\
&|C_T|\ar[d]\\
&T
}
\]
and isomorphisms $\varphi_i:T\times B\mu_{r_i}\to (\Sigma_T)_i$ for $i=1,\dots n$ where \begin{enumerate}
\item $C_T$ is a Deligne--Mumford stack, proper over $T$ and \'{e}tale locally a nodal curve over $T$.
\item $(\Sigma_T)_i\subset C$ for $i=1,\dots,n$ are disjoint closed substacks in the smooth locus of $C_T\to T$. 
\item $(\Sigma_T)_i\to T$ is a gerbe banded by $\mu_{r_i}$ (the order $r_i$ is well-defined since $T$ is connected) trivialized by $\varphi_i$.
\item $|C_T|$ is the coarse moduli of $C_T$, a nodal curve over $T$ isomorphic to $C_T$ away from nodes and markings
\item The action of the stabilizer group at the nodes is balanced.
\end{enumerate}
We will denote an object of $\fM^{\tw}(\bullet)$ by $[(C,\Sigma_1,\dots ,\Sigma_n)]$.
 Then $\fM^{\tw}$ is a smooth Artin stack of dimension $3g-3+n$ and the morphism $\fM^{\tw}\to\fM$ fogetting the orbifold structure is a DM stack locally of finite type \cite[Theorem 1.10, Corollary 1.12]{O07}, \cite[Proposition 2.2.1]{AJ03}. $\fM^{\tw}$ decomposes into open and closed substacks indexed by $(r_1,\dots,r_n)$, $n$-tuples recording the orders of the automorphism groups at the marked points. 
Let $\fC_{\fM^{\tw}}\to \fM^{\tw}$ be the universal curve.

We can repeat the constructions of the previous sections to define the moduli space 
\[
\fB^{\tw}:=\Bun^{\tw}_{g,n,r,d}
\] of twisted genus $g$ $n$-pointed curves together with a rank $r$ vector bundle of degree $d\in\bbQ$. This is a smooth Artin stack, but it has components of various dimensions. To describe them, we recall the notion of the \textit{age} of a vector bundle over a DM stack. 
\begin{definition}\label{def:age}
Let $\ccF$ be a locally-free sheaf on a DM stack $\ccY$,and $q:I\ccY\to\ccY$ be the inertia stack of $\ccY$. Let $(y,g)\in I\ccY$. Since $g$ is finite, the action of $g$ on the fiber $q^*\ccF$ is diagonalizable and decomposes it into a direct sum $\bigoplus_{0\leq k< j}q^*\ccF_{(y,g)}^{(k)}$, where $g$ acts on $q^*\ccF_{(y,g)}^{(k)}$ by $e^{2\pi\sqrt{-1}k/j}$. The age of $\ccF$ at $(y,g)$ is defined as
\[
\age_{(y,g)}(\ccF)=\frac{1}{j}\sum_{0\leq k< j}k\cdot\mathrm{dim}(q^*\ccF_{(y,g)}^{(k)}).
\]
\end{definition}
Note that with our definitions, a marked point gives an object of the inertia stack $I C$ by composing its inclusion map with $\varphi_i$, $[B\mu_{r_i}\cong\Sigma_i\to C]\in\mathrm{Ob}(IC)$. So it makes sense to speak of the age of a vector bundle at a marked point. 
With this definition, we have the following version of Riemann--Roch for orbi-curves.
\begin{theorem}\cite[Theorem 7.2.1]{AGV08}
Let $(C,\Sigma_1,\dots,\Sigma_n)$ be a balanced twisted curve over $\bbC$ of genus $g$ and $\ccF$ a locally-free sheaf of rank $r$, then
\[
\chi(\ccF)=r(1-g)+\deg(\ccF)-\sum_{i=1}^n\age_{\Sigma_i}(\ccF).
\]
\end{theorem} 
Then the relative dimension of $\fB^{\tw}\to\fM^{\tw}$ at a point $[(C,\Sigma_1,\dots,\Sigma_n),\ccF]$  is $-\chi(\End(\ccF))$, which depends on the age of the bundle at the marked points. Say $\Sigma_i\cong B\mu_{r_i}$, the generator $\zeta_{r_i}$ of $\mu_{r_i}$ acts diagonalizably on the fiber of $\ccF$ with eigenvalues $e^{2\pi\sqrt{-1}k/{r_i}}$ for $0\leq k<r_i$ each with multiplicity $m_{k}$ (possibly zero), so that \[\age_{\Sigma_i}(\ccF)=\sum_{0\leq k< r_i}\frac{k}{r_i}m_{k}.\] Note that $(m_0,\dots,m_{r_i-1})$ gives a partition of $r$. We can see that 
\begin{equation}\label{dimflag}
\age_{\Sigma_i}(\End(\ccF))=\sum_{0\leq k<l<r_i} m_{k} m_l=\dim(F_{m_0,\dots,m_{r_i-1}})
\end{equation}
where $F_{m_0,\dots,m_{r_i-1}}$ is the partial flag variety defined by the given partition.

This age only depends on the conjugacy class $a_i$ defined by the image of the generator $\zeta_{r_i}$. Fix an n-tuple $\underline{a}:=(a_1,\dots ,a_n)$, we can define an open and closed substack $\fB^{tw}(\underline{a})$ of $\fB^{\tw}$ of constant dimension by fixing the representations of the universal marked points on the universal bundle. 
Indeed, having fixed a trivialization of the gerbes gives evaluation maps at the marked points $ev_i: \fB^{tw}\to IBGL_r$ the inertia stack of $BGL_r$ has components indexed by conjugacy classes $IBGL_r=\bigsqcup (BGL_r)_a$. We are decomposing $\fB^{\tw}$ according to the image of the evaluation maps at the marked points. 

The dimension of $\fB^{\tw}(\underline{a})$ is 
\begin{align*}
d_{\fB^{\tw}(\underline{a})}&=3g-3+n-(r^2(1-g)-\sum_{i=1}^n \age_{\Sigma_i}(\End(\ccF))\\
&= (3+r^2)(g-1)+n+\sum_{i=1}^n \dim(F_{m_0,\dots,m_{r_i-1}}).
\end{align*}

Finally, let $\ffH^{\tw}:=\pi_{\fB^{\tw}*}\omega_{\pi_{\fB^{\tw}}}$ denote the Hodge bundle over $\fB^{\tw}$, the projection $\ffH^{\tw}\to\fB^{\tw}$ is a vector bundle of rank $g$.

\subsection{Moduli of maps to $\ccX$}\label{curvestoX}
Fix the triple $g,n,\beta$. The moduli of stable twisted maps of degree $\beta\in H_2(X,\bZ)$ to $\ccX$ with sections of all the gerbes is the moduli space $\fX^{\tw}=\fM_{g,n}(\ccX,\beta)\subset\Hom_{\ccX}(\fC_{\fM^{\tw}},\ccX)$. A map from a twisted curve $[(C,\Sigma_1,\dots ,\Sigma_n)]\in\fM^{\tw}$ is a representable morphism
\[
f:(C,\Sigma_1,\dots,\Sigma_n)\to \ccX
\]
such that the underling morphism of coarse moduli spaces $|f|:|C|\to X$ has degree $\beta$ and the automorphism group $\Aut_{\ccX}(f,\Sigma_i)$ of $f$ fixing the marked points is finite. 

$\fX^{\tw}$ is a proper stack with projective coarse moduli space, the morphism $\fX^{\tw}\to\fX=\fM_{g,n}(X,\beta)$ is of finite type. Let, in keeping with our notation so far, $\pi_{\fX^{\tw}}: \fC_{\fX^{\tw}}\to \fX^{\tw}$ be the universal curve over $\fX^{\tw}$, and let $ev_{\fX^{\tw}}: \fC_{\fX^{\tw}}\to \ccX$ be the universal evaluation map. 
 There is a perfect obstruction theory relative to $\fM^{\tw}$ given by
\begin{equation}\label{potXtwMtw}
\phi_{\fX^{\tw}/\fM^{\tw}}:\bE_{\fX^{\tw}/\fM^{\tw}}=\left(R^\bullet \pi_{\fX^{\tw}*}ev^*_{\fX^{\tw}}T_{\ccX} \right)^{\vee}.
\end{equation}
The virtual dimension of $\fX^{\tw}$ is more subtle to compute than that of its counterpart $\fX$, as the orbifold version of the Riemann-Roch formula takes into account the age of marked points.

\begin{definition}\cite[Definition 2.3.3]{Tseng}
The age of a substack $\ccX_{\lambda}\subset I\ccX$ is defined as
\[
\age(\ccX_\lambda)=\age_{(x,g)}T_{\ccX}
\]
for any object $(x,g)$ of $\ccX_\lambda(\bullet)$. It is independent of the chosen representative. 
\end{definition}
\begin{remark}
More generally, the age of any vector bundle on $\ccX$ is constant on connected components of the inertia stack, so we will use notation $\age_{\ccX_{\lambda}}(\ccE)$.
\end{remark}

The universal family of $\fX^{\tw}$ includes the data of evaluation maps at the marked points:
\[
ev_i:\fX^{\tw}\to I\ccX=\bigsqcup_{\lambda\in\bbN}\ccX_\lambda. 
\]
These allow to decompose the $\fX^{\tw}$ into open and closed substacks
\[
\fX^{\tw}=\bigsqcup_{\lambda_1,\dots,\lambda_n}\fX^{\tw}(\lambda_1,\dots,\lambda_n)
\]
with
\[
\fX^{\tw}(\lambda_1,\dots,\lambda_n)=\bigcap_{i=1}^nev_i^{-1}(\ccX_{\lambda_i}).
\]

Let $[(C,\Sigma_1,\dots,\Sigma_n),f]$ an object of $\fX^{\tw}(\lambda_1,\dots,\lambda_n)$, we have
\begin{align*}
d^{\vir}_{\fX^{\tw}(\lambda_1,\dots,\lambda_n)/\fM^{\tw}}&=\chi(f^*T_{\ccX})\\
&= \dim(\ccX)(1-g)+\deg(f^*T_{\ccX})-\sum_{i=1}^n \age_{\Sigma_i}(f^*T_{\ccX})\\
&= \dim(\ccX)(1-g)+\deg(f^*T_{\ccX})-\sum_{i=1}^n \age(\ccX_{\lambda_i})
\end{align*}
The perfect obstruction theory in \cref{potXtwMtw} restricts to each open and closed substack $\fX^{\tw}(\lambda_1,\dots,\lambda_n)$ and we can define $[\fX^{\tw}/\fM^{\tw}]^{\vir}$ as
\[
[\fX^{\tw}/\fM^{\tw}]^{\vir}=\sum_{\lambda\in\Lambda}[\fX^{\tw}(\lambda_1,\dots,\lambda_n)/\fM^{\tw}]^{\vir}\in A_*(\fX^{\tw}).
\]
Where the cycles on the right-hand-side are understood as pushed forward to $ A_*(\fX^{\tw})$. This is now a linear combination of cycles in different dimensions.

\subsubsection{A perfect obstruction theory for $\pi_{\fX^{\tw}/\fB^{\tw}}: \fX^{\tw}\to \fB^{\tw}$}\label{curvestoX/B}

Having fixed the sheaf $\ccE$ gives a forgetful morphism 
\[
\fX^{\tw}\to \fB^{\tw}
\]
where the degree $d$ implicit in the notation of the second stack is 
\[
d=\frac{1}{m}\int_{\beta}c_1(L)
\]
where $L$ is a line bundle on $X$ such that $det(\ccE)^{\otimes m}=\pi_{\ccX/X}^*L$.

As in \cref{Tor1} have the following commutative diagram

\begin{equation}
\xymatrix{
\fC_{\fX^{\tw}}\ar[d]_{\overline{\pi_{\fX^{\tw}/\fB^{\tw}}}}\ar[r]^{ev_{\fX^{\tw}}} & \ccX\ar[d]^{\psi_\ccE}\\
\fC_{\fB^{\tw}}\ar[r]^{ev_{\fB^{\tw}}} & BGL_r
}
\end{equation}
which induces a morphism of cotangent complexes
\[
ev_{\fX^{\tw}}^*\bL_{\ccX/BGL_r} \to\bL_{\fC_{\fX^{\tw}}/\fC_{\fB^{\tw}}}=\pi_{\fX}^*\bL_{\fX^{\tw}/\fB^{\tw}}.
\]
Dualizing and pushing forward to we obtain:
\[
\phi_{\fX^{\tw}/\fB^{\tw}}^\vee: \bT_{\fX^{\tw}/\fB^{\tw}}= R^{\bullet}\pi_{\fX^{\tw}*}\pi_{\fX^{\tw}}^*\bT_{\fC_{\fX^{\tw}}/\fC_{\fB^{\tw}}}\to R^{\bullet}\pi_{\fX^{\tw}*}ev_{\fX^{\tw}}^*A_\ccE.
\]
 
\begin{proposition}
The morphism 
\begin{equation}\label{potXtwBtw}
\phi_{\fX^{\tw}/\fB^{\tw}}:  \bE_{\fX^{\tw}/\fB^{\tw}}=(R^{\bullet}\pi_{\fX^{\tw}*}ev_{\fX^{\tw}}^*A_{\ccE})^{\vee}\to\bL_{\fX^{\tw}/\fB^{\tw}}.
\end{equation}
is a perfect obstruction theory for $\pi_{\fX^{\tw}/\fB^{\tw}}$. 

\end{proposition}
\begin{proof}
This is identical to the proof of \cref{potX}.
\end{proof}
We can now construct a virtual fundamental class $[\fX^{\tw}/\fB^{\tw}]^{\vir}$ in the same way as $[\fX^{\tw}/\fM^{\tw}]^{\vir}$. 

Consider $\psi_{\ccE}:\ccX\to BGL_r$, the corresponding map $I\psi_{\ccE}$ between inertia stacks sends each component $\ccX_{\lambda_i}$ to a component $(BGL_r)_{a_i}$ for some conjugacy class $a_i$. Let $a(\underline{\lambda})$ denote the $n$-tuple of conjugacy classes of $GL_r$ obtained from an $n$-tuple $\lambda_1,\dots,\lambda_n$. We have a commutative diagram 
\[
\xymatrix{
\fX^{\tw}\ar[r]^{ev_i}\ar[d] & I\ccX\ar[d]^{I\psi_{\ccE}}\\
\fB^{\tw}\ar[r]^{ev_i} & IBGL_r.
}
\]
This shows that the image of each component $\fX^{\tw}(\lambda_1,\dots,\lambda_n)$ lies in $\fB^{\tw}(a(\underline{\lambda}))$. 

The relative perfect obstruction theory in \cref{potXtwBtw} restricts to the substacks $\fX^{\tw}(\lambda_1,\dots,\lambda_n)\to\fB^{\tw}(a(\underline{\lambda}))$, which each have constant relative virtual dimension

\begin{align*}
d^{\vir}_{\fX^{\tw}(\lambda_1,\dots,\lambda_n)/\fB^{\tw}a(\underline{\lambda})}&=\chi(f^*A_{\ccE})\\
&= (\dim(\ccX)+r^2)(1-g)+\deg(f^*A_{\ccE})-\sum_{i=1}^n \age_{\Sigma_i}(f^*A_{\ccE})\\
&= (\dim(\ccX)+r^2)(1-g)+\deg(f^*T_{\ccX})-\sum_{i=1}^n \age(\ccX_{\lambda_i})+\age_{\ccX_{\lambda_i}}(\End(\ccE))\\
&= d^{\vir}_{\fX^{\tw}(\lambda_1,\dots,\lambda_n)/\fM^{\tw}}+\dim(\fM^{\tw})-\dim(\fB^{\tw}(a(\underline{\lambda}))).
\end{align*}
where the second line follows by the short exact sequence defining $A_{\ccE}$ and the last by the observation that 
\[
\age_{\ccX_{\lambda_i}}(\End(\ccE))=\age_{\Sigma_i}(\End(\ccE))
\]
is the same age appearing in \cref{dimflag}.
So the stacks $\fX^{\tw}(\lambda_1,\dots,\lambda_n)$, with either of the perfect obstruction theories introduced so far, have absolute virtual dimension
\[
d^{\vir}_{\fX^{\tw}(\lambda_1,\dots,\lambda_n)}=(\dim(\ccX)-3)(1-g)+\deg(f^*T_{\ccX})-\sum_{i=1}^n (\age(\ccX_{\lambda_i})-1).
\]

Then define
\[
[\fX^{\tw}/\fB^{\tw}]^{\vir}=\sum_{(\lambda_1,\dots,\lambda_n)}[\fX^{\tw}(\lambda_1,\dots,\lambda_n)/\fB^{\tw}]^{\vir}\in A_*(\fX^{\tw}).
\]
We know that the cycles $[\fX^{\tw}(\lambda_1,\dots,\lambda_n)/\fM^{\tw}]^{\vir}$ and $[\fX^{\tw}(\lambda_1,\dots,\lambda_n)/\fB^{\tw}]^{\vir}$ have the same virtual dimension, it remains to check that they agree in the same way as \cref{potXcompare}.

\begin{proposition} With the definitions above we have 
\[
[\fX^{\tw}(\lambda_1,\dots,\lambda_n)/\fB^{\tw}]^{\vir}=[\fX^{\tw}(\lambda_1,\dots,\lambda_n)/\fM^{\tw}]^{\vir}\in A_{d^{\vir}_{\fX^{\tw}(\lambda_1,\dots,\lambda_n)}}( \fX^{\tw})
\]
for each set of indices $(\lambda_1,\dots,\lambda_n)$, so 
\[
[\fX^{\tw}/\fM^{\tw}]^{\vir}=[\fX^{\tw}/\fB^{\tw}]^{\vir}\in A_*(\fX^{\tw}).
\]

\end{proposition}

\begin{proof}
The proof is as in \cref{potXcompare}, replacing $\fX$ by $\fX^{\tw}(\lambda_1,\dots,\lambda_n)$ and $\fM$ and $\fB$ by their respective twisted versions. 
\end{proof}

\subsubsection{Moduli of maps to $\ccX$ with fields}
Let
\[
\fX^{\ccE} ={\Spec}_{\fX^{\tw}}\Sym \,R^1\pi_{\fX^{\tw}*}(ev_{\fX^{\tw}}^*\ccE)
\]

be the moduli of genus $g$, $n$-pointed, degree $\beta$ twisted maps to $\ccX$ with a $p$-field. In the notation of \cite[Section 2]{CL}, this is the direct image cone
\[
\fX^{\ccE} =C\left(\pi_{\fX^{\tw}*}\left(ev_{\fX^{\tw}}^*\ccE^{\vee}\otimes \omega_{\pi_{\fX^{\tw}}}\right)\right)
\] 
 Objects of $\fX^{\ccE}(\bullet)$ are triples $[(C,\Sigma_1,\ldots,\Sigma_n), f, p]$ where $[(C,\Sigma_1,\ldots,\Sigma_n), f]$ is an object in $\fX^{\tw}(\bullet)$ and
$p\in H^0(C,f^*\ccE^\vee\otimes\omega_C)$. 

From its definition, we see that $\fX^{\ccE}$ has a dual relative perfect obstruction theory over $\fX^{\tw}$ given by 
\[
\bE_{\fX^{\ccE}/\fX^{\tw}}^\vee = R^\bullet\pi_{\fX^{\ccE} *}\left(ev_{\fX^{\ccE}}^*\ccE^{\vee}\otimes \omega_{\pi_{\fX^{\ccE}}}\right).
\]
We can similarly decompose $\fX^{\ccE}$ into 
\[
\fX^{\ccE}=\bigsqcup_{(\lambda_1,\dots,\lambda_n)}\fX^{\ccE}(\lambda_1,\dots,\lambda_n)
\]
 with each component being the preimage of $\fX^{\tw}(\lambda_1,\dots,\lambda_n)$.
Then the relative virtual dimension of the graded pieces is constant
\begin{align*}
d^{\vir}_{\fX^{\ccE}(\lambda_1,\dots,\lambda_n)/\fX^{\tw}(\lambda_1,\dots,\lambda_n)}&=\chi(f^*\ccE^{\vee}\otimes\omega_C)\\
&=-\chi(f^*\ccE)\\
&=r(g-1)-\deg(f^*\ccE)+\sum_{i=1}^n\age_{\ccX_{\lambda_i}}(\ccE)
\end{align*}
The absolute virtual dimension is then
\[
d^{\vir}_{\fX^{\ccE}(\lambda_1,\dots,\lambda_n)}=(\dim(\ccX)-r-3)(1-g)+\deg(f^*T_{\ccX})-\deg(f^*\ccE)-\sum_{i=1}^n\age_{\ccX_{\lambda_i}}(T_{\ccX})-\age_{\ccX_{\lambda_i}}(\ccE)-1.
\]
For the comparison results, we will need to consider the relative perfect obstruction theory of $\fX^\ccE\to\fB^{\tw}$. The situation is analogous to \cref{ss:fields}, so we have $\phi_{\fX^\ccE/\fB^{\tw}}:\bbE_{\fX^\ccE/\fB^{\tw}}\to\bbL_{\fX^\ccE/\fB^{\tw}}$ where
\[
\bbE_{\fX^\ccE/\fB^{\tw}}=\left(R^\bullet\pi_{\fX^\ccE*}\left(ev_{\fX^\ccE}^*A_\ccE\oplus ev_{\fX^\ccE}^*\cE^\vee\otimes\omega_{\pi_{\fX^\ccE}}\right)\right)^{\vee}.
\]

\subsubsection{The cosection}\label{sec:twcos}
The section $s$ induces a morphism $\ccW$ over $\fB^{\tw}$:
\[
\xymatrix{
\fX^{\ccE}\ar[r]^{\ccW}\ar[dr]& \ffH^{\tw}\ar[d]\\
& \fB^{\tw}}
\]
The local description of $\ccW$ is
\[
[(C,\Sigma_1,\dots,\Sigma_n),f,p]\mapsto[(C,\Sigma_1,\dots\Sigma_n), p\otimes f^*s]
\]
 where $p\otimes f^*s\in H^0(C,f^*\ccE\otimes f^*\ccE^{\vee}\otimes\omega_C)$ is interpreted as a section of $\omega_C$. Globally, this morphism comes from the restriction of the morphism between direct image cones.
The relative differential of this morphism is
\begin{equation}\label{dW}
d\ccW:\bbT_{\fX^{\ccE}/\fB^{\tw}}\to\ccW^*T_{\ffH/\fB^{\tw}}.
\end{equation}
On the other hand, we have a perfect obstruction theory $\phi_{\ffH^{\tw}/\fB^{\tw}}$ for $\ffH^{\tw}\to\fB^{\tw}$ given by
\[
\bbE_{\ffH^{\tw}/\fB^{\tw}}=\left(R^{\bullet}\pi_{\ffH^{\tw}*}\omega_{\pi_{\ffH^{\tw}}}\right)^{\vee}.
\]
Here the virtual dimension is $g-1$ because the obstruction sheaf, which here is just a trivial line bundle, is too large. 

We can view $\ccW$ as inducing a morphism (c.f. \cref{cosecXEB}):
\begin{equation}\label{deltaW}
\delta \ccW: \bbE_{\fX^{\ccE}/\fB^{\tw}}^{\vee}\to\ccW^*\bbE_{\ffH^{\tw}/\fB^{\tw}}^{\vee}.
\end{equation}
The relative cosection is
\[
\sigma_{\fX^{\ccE}/\fB^{\tw}}=h^1(\delta\ccW):\Ob_{\fX^{\ccE}/\fB^{\tw}}\to\ccO_{\fX^{\ccE}}.
\]
The cosection factors as an absolute cosection
\[
\sigma_{\fX^{\ccE}}:\Ob_{\fX^{\ccE}}\to\cO_{\fX^{\ccE}}
\]
by the observation that \cref{dW} and \cref{deltaW} commute with $\phi_{\fX^{\ccE}/\fB^{\tw}}^{\vee}$ and $\phi_{\ffH^{\tw}/\fB^{\tw}}^{\vee}$. Then $h^1(\phi_{\fX^{\ccE}/\fB^{\tw}}^{\vee}\circ\delta\ccW)=0$, which is a sufficient condition for the cosection to factor by \cref{ss:lift}.

\begin{proposition}\label{kjgt}
The relative cosection $\sigma_{\fX^{\ccE}/\fB^{\tw}}$ has local expression
\begin{align*}
\sigma_{\fX^{\ccE}/\fB^{\tw}}|_{\xi}: & H^1(f^*A_\ccE)\oplus H^1(f^*\ccE^\vee\otimes\omega_C)\to\bC\\
&(\dot{z},\dot{p}) \mapsto \langle f^*\delta s(\dot{z}), p\rangle +\langle f^*s,\dot{p}\rangle
\end{align*}
for $\xi=[(C,\Sigma_1,\dots\Sigma_n),f,p]$.
\end{proposition}

\subsubsection{Moduli of maps to \ccZ}
Let $i:\ccZ\to\ccX$ and consider the moduli of twisted stable maps to $\ccZ$ with sections of the marked gerbes of degrees $\beta'$ such that $|i|_*\beta'=\beta\in H_2(X,\bbZ)$ for $|i|:Z\xhookrightarrow{}X$. 
\[
\fZ^{\tw}=\bigsqcup_{\beta'}\fM_{g,n}(\ccZ,\beta')\xrightarrow{\iota}\fX^{\tw}.
\]
There are universal evaluation maps at the markings
\[
ev_i:\fZ^{\tw}\to I\ccZ
\]
and we can decompose it into open and closed substacks
\[
\fZ^{\tw}(\lambda_1,\dots,\lambda_n)=\bigcap_{i=1}^nev_i^{-1}(\ccZ_{\lambda_i}).
\]
And $\iota:\fZ^{\tw}\to\fX^{\tw}$ restricts to $\fZ^{\tw}(\lambda_1,\dots,\lambda_n)\to\fX^{\tw}(\lambda_1,\dots,\lambda_n)$. Let $\ccN=i^*\ccE$, fixing this bundle gives $\fZ^{\tw}\to\fB^{\tw}$ with perfect obstruction theory $\phi_{\fZ^{\tw}/\fB^{\tw}}$
\[
\bE_{\fZ^{\tw}/\fB^{\tw}}=\left(R^{\bullet}{\pi_{\fZ^{\tw}*}}ev_{\fZ^{\tw}}^*(A_{\ccN})\right)^{\vee}.
\]
The virtual dimension is constant on each $\fZ^{\tw}(\lambda_1,\dots,\lambda_n)$ and equal to the one coming from the familiar perfect obstruction theory $\bE_{\fZ^{\tw}/\fB^{\tw}}=(R^{\bullet}{\pi_{\fZ^{\tw}*}}ev_{\fZ^{\tw}}^*T_{\ccZ})^\vee$, it is given by

\[
d^{\vir}_{\fZ^{\tw}(\lambda_1,\dots,\lambda_n)}=\dim(\ccZ)(1-g)+\deg(f^*T_{\ccZ})-\sum_{i=1}^n\age_{\ccZ_{\lambda_i}}(T_{\ccZ})-1.
\]
All this follows from the discussion in \cref{curvestoX} and \cref{curvestoX/B}. The short exact sequence of sheaves on $\ccZ$
\[
0\to T_{\ccZ}\to i^*T_{\ccX}\to i^*\ccE\to 0
\]
implies that
\[
d^{\vir}_{\fZ^{\tw}(\lambda_1,\dots,\lambda_n)}=d^{\vir}_{\fX^{\ccE}(\lambda_1,\dots,\lambda_n)}.
\]
From \cref{kjgt} we see that
\begin{proposition}
The cosection $\sigma_{\fX^{\ccE}/\fB^{\tw}}$ restricted to the substack $\fX^{\ccE}(\lambda_1,\dots,\lambda_n)$ has degeneracy locus
\[
\fZ^{\tw}(\lambda_1,\dots,\lambda_n).
\]
\end{proposition}
We then have two virtual fundamental classes for ${\fZ^{\tw}(\lambda_1,\dots,\lambda_n)}$ of the same degree. The usual virtual fundamental class
\[
[\fZ^{\tw}(\lambda_1,\dots,\lambda_n)]^{\vir}:=[\fZ^{\tw}(\lambda_1,\dots,\lambda_n)/\fB^{\tw}]=[\fZ^{\tw}(\lambda_1,\dots,\lambda_n)/\fM^{\tw}]
\]
and the cosection-localized class of $\fX^{\ccE}(\lambda_1,\dots,\lambda_n)$
\[
[\fX^{\ccE}(\lambda_1,\dots,\lambda_n)]^{\vir}_{\sigma_{\fX^{\ccE}}}.
\]

\begin{theorem}\label{mainthmDM}
For each indexing set $(\lambda_1,\dots,\lambda_n)$ we have
\[
[\fX^{\ccE}(\lambda_1,\dots,\lambda_n)]^{\vir}_{\sigma_{\fX^{\ccE}}}=(-1)^{r(1-g)+d-\sum_{i=1}^n\age_{\ccX_{\lambda_i}}(\ccE)}[\fZ^{\tw}(\lambda_1,\dots,\lambda_n)]^{\vir}
\]
\end{theorem}
The proof of this theorem proceeds as usual via degeneration to the normal cone and virtual localization. The arguments go without major modifications, but we summarize the statements below in the notation of this section.

\subsection{Degeneration to the normal cone}
We can consider the total space of the deformation to the normal cone of $\ccZ$ in $\ccX$, namely
 \[
\ccV=Bl_{\ccZ\times \{0\}}(\ccX\times\bA^1)\setminus\widetilde{\ccX\times\{0\}}.
\]
This is a DM stack over $\bA^1$ with generic fiber isomorphic to $\ccX$ and special fiber $\Vb(\ccN)$. Let $\pi_{\ccV/\ccX}$ denote the projection to $\ccX$ and fix the locally-free sheaf $\tilde{\ccE}=\pi_{\ccV/\ccX}^*\ccE$. Over $t\neq 0$, this is the vector bundle $\ccE$ on $\ccX\cong\ccV|_{t}$. Over $0\in\bA^1$, this is $\pi^*\ccN$ over $\Vb(\ccN)$ where $\pi:\Vb(\ccN)\to\ccZ$.
Fix the section $y$ to be $y=t^{-1}s$ over $t\neq 0$ and, over 0, the tautological section given by the identity of $\ccN$. This is a regular section with vanishing locus $Z(\tilde{s})=\ccZ\times\bA^1$ (c.f. \cref{deglocy}). 

\begin{proposition}[The inertia stack of $\ccV$]\label{inertiaV}
The inertia stack of $\ccV$ decomposes into disjoint components
\[
I\ccV=\bigsqcup_{\lambda\in\Lambda}\ccV_{\lambda}
\]
where $\Lambda$ is the same finite indexing set indexing the inertia stack of $\ccX$. For each component, we have a projection $\pi:\ccV_{\lambda}\to\bA^1$ such that $\pi^{-1}(t)=\ccX_{\lambda}$ and $\pi^{-1}(0)=\ccZ_{\lambda}$.
\end{proposition}
\begin{proof}
Let $\pi$ be the composition $I\ccV\to\ccV\to\bA^1$. We claim that the left-hand square in the following diagram is fibered
\[
\xymatrix{
I\ccV\ar[r]&\ccV\ar[r] & \bA^1\\\
I\ccV_t\ar[u]\ar[r] & \ccV_t\ar[u]\ar[r] & t\ar[u].
}
\]
This follows by the fact that the inclusion of the fiber $\ccV_t\to\ccV$ is a fully faithful morphism. Then by \cite[\href{https://stacks.math.columbia.edu/tag/06R5}{Lemma 06R5}]{stacks-project} the left square is fibered. In practice, for a point $x\in\ccV_t$ the only automorphisms of $x\in\ccV$ are those that come from $\ccV_t$. 
\end{proof}

The moduli of twisted stable maps with fields defined by $(\ccV,\tilde{\ccE})$ is $\ccV^{\tilde{\ccE}}$. This is a family over $\bA^1$ with fiber $\ccX^{\ccE}$ over $t\neq0$ and fiber over $0\in\bA^1$ given by
\[
\fN^{\ccN}=\Spec_{\fZ^{\tw}}\Sym(R^1\pi_{\fZ^{\tw}*}ev_{\fZ^{\tw}}^*\ccN\oplus ev_{\fZ^{\tw}}^*\ccN^{\vee}\otimes\omega_{\pi_{\fZ^{\tw}}} ). 
\]
The latter is a cone of virtual dimension $0$ over $\fZ^{\tw}$, with a symmetric perfect obstruction theory. We can decompose it into substacks $\fN^{\ccN}(\lambda_1,\dots,\lambda_n)$ each supported over $\fZ^{\tw}(\lambda_1,\dots,\lambda_n)$.

\begin{lemma}\label{2potsofV}
The moduli of stable maps with fields to $(\ccV,\tilde{\ccE})$ has a perfect obstruction theory relative to $\fB^{\tw}\times\bA^1$ given by
\[
\bbE_{\ffV^{\tilde{\ccE}}/\fB^{\tw}\times\bA^1}=R^{\bullet}\pi_{\ffV^{\tilde{\ccE}}}\left(ev_{\ffV^{\tilde{\ccE}}}^*A\oplus ev_{\ffV^{\tilde{\ccE}}}^*\tilde{\ccE}^{\vee}\otimes\omega_{\pi_{{\ffV^{\tilde{\ccE}}}}}\right)
\]
where $A=T_{\ccV/BGL_r\times\bA^1}$.
\end{lemma}
\begin{proof}
This is exactly as \cref{VrelA} and the discussion in that section.
\end{proof}

The section $y\in\Gamma(\ccV,\tilde{\ccE})$ defines a cosection of the obstruction sheaf. We can introduce a perfect obstruction theory $\bE'_{\fV^{\tilde{\ccE}}/\fB^{\tw}}$ compatible with $\bbE_{\ffV^{\tilde{\ccE}}/\fB^{\tw}\times\bA^1}$, as in \cref{ss:deform}. The cosection lifts as usual, so we obtain an absolute cosection $\sigma_{\fV^{\tilde{\ccE}}}$.
 
Now that we have the perfect obstruction theory and cosection, we can break $\fV^{\tilde{\ccE}}$ into substack of constant virtual dimension to construct the virtual fundamental class. This decomposition is compatible by restriction with the decompositions of $\fX^{\ccE}$ and $\fN^{\ccN}$ because of \cref{inertiaV} and because the evalaution maps commute with the inclusions:
\[\xymatrix{
\fX^{\ccE}\ar[r]^{ev_i} \ar[d]& I\ccX\ar[d]\\
\fV^{\tilde{\ccE}}\ar[r]^{ev_i} & I\ccV.}
\]

Let 
\[
\fV^{\tilde{\ccE}}(\lambda_1,\dots\lambda_n):=\bigsqcup_{i=1}^nev_i^{-1}(I\ccV).
\]
On each such substack, we have a cosection-localized virtual fundamental class
\[
[\fV^{\tilde{\ccE}}(\lambda_1,\dots\lambda_n)]^{\vir}_{\sigma}\in A_*(\fZ(\lambda_1,\dots\lambda_n)\times\bA^1)
\]
which can be constructed using $\bbE_{\ffV^{\tilde{\ccE}}/\fB^{\tw}\times\bA^1}$ and the (absolute) cosection constructed from $y$ following the recipe of \cref{sec:twcos}.
\begin{theorem}\label{definvDMcase}
\[
\iota_0^![\fV^{\tilde{\ccE}}(\lambda_1,\dots\lambda_n)]^{\vir}_{\sigma_{\fV^{\tilde{\ccE}}}}=\iota_1^![\fV^{\tilde{\ccE}}(\lambda_1,\dots\lambda_n)]^{\vir}_{\sigma_{\fV^{\tilde{\ccE}}}}=[\fX^{\ccE}(\lambda_1,\dots\lambda_n)]^{\vir}_{\sigma_{\fX^{\ccE}}}\in A_*(\fZ^{\tw})
\]
\end{theorem}
\begin{proof}
The first equality follows by deformation-invariance of the cosection localized virtual fundamental class, as per \cref{deformationthm}. This requires introducing the two perfect obstruction theories of $\fV^{\tilde{\ccE}}$ and checking their compatibility. The second equality follows by checking that
\begin{enumerate}
\item
 $\iota_1^*\bbE_{\fV^{\tilde{\ccE}}/\fB^{\tw}\times\bA^1}=\bbE_{\fX^{\ccE}/\fB^{\tw}}$
 \item $A\to A_{\tilde{\ccE}}\xrightarrow{\delta y}\tilde{\ccE}$ restricts to $\delta s:A_{\ccE}\to\ccE$ over $\ccV_1$.
 \end{enumerate}
The first point follows from the fact that $A|_{\ccV_1}\cong A_E$ and the second from the observation that at $t=1$, $y=\pi_{\ccV/\ccX}^*s$.
\end{proof}
\subsection{Torus localization}
Finally, to prove \cref{mainthmDM} we need to establish the following result by torus localization:
\[
\iota_0^![\fV^{\tilde{\ccE}}(\lambda_1,\dots\lambda_n)]^{\vir}_{\sigma_{\fV^{\tilde{\ccE}}}}=(-1)^{r(1-g)+d-\sum_{i=1}^n\age_{\ccX_{\lambda_i}}(\ccE)}[\fZ^{\tw}(\lambda_1,\dots,\lambda_n)]^{\vir}
\]

We first identify $\iota_0^![\fV^{\tilde{\ccE}}(\lambda_1,\dots\lambda_n)]^{\vir}_{\sigma}$ with a cosection localized class on $\fN^{\ccN}(\lambda_1,\dots,\lambda_n)$. Let the evaluation map of $\fN^{\ccN}$ to $\ccZ$ be denoted by $\ev$ in what follows, then on any connected component of $\fN^{\ccN}$ the relative perfect obstruction theory to $\fB^{\tw}$ is given by 
\[
\bbE_{\fN^{\ccN/\fB^{\tw}}}=\left( R^\bullet\pi_{\fN^{\ccN}*}\ev^*(A_{\ccN}\oplus\ccN)\oplus ev^*\ccN^{\vee}\otimes\omega_{\pi_{\fN^{\ccN}}} \right)^{\vee}.
\]
A point in $\fN^{\ccN}$ is $\xi=[(C,\Sigma_1,\dots\Sigma_n),f,p,q]$ with $[(C,\Sigma_1,\dots\Sigma_n),f]\in\fZ^{\tw}$ and $p\in H^0(C,f^*\fN^{\vee}\otimes\omega_C)$, $q\in H^0(C,f^*\ccN)$.
The cosection coming from the restriction deformation of $\sigma_{\ffV^{\tilde{\ccE}}/\fB^{\tw}\times\bA^1}$ is just 
\begin{align*}
\sigma_{\fN^{\ccN}/\fB^{\tw}}|_{\xi}&:H^1(C,f^*A_{\ccN})\oplus H^1(C,\ccN))\oplus H^1(C,f^*\ccN^{\vee}\otimes\omega_C)\to \bbC\\
&(\dot{x},\dot{q},\dot{p})\mapsto \langle\dot{p},q\rangle+\langle p,\dot{q}\rangle.
\end{align*}
This cosection is equivariant with respect to the $\bC^*$-action induced on $\fN^{\ccN}$ by rescaling the fibers of $\ccN$ with opposite weights. The torus action also preserves the components $\fN^{\ccN}(\lambda_1,\dots,\lambda_n)$, and $(\fN^{\ccN}(\lambda_1,\dots,\lambda_n))^{\bC^*}=\fZ(\lambda_1,\dots,\lambda_n)$. At the level of perfect obstruction theories, the situation is precisely analogous to \cref{section:localization}. We have $\bE_{\fN^{\ccN}/\fB^{\tw}}^{\fix}|_{\fZ^{\tw}}=\bE_{\fZ^{\tw}/\fB^{\tw}}$ and the induced cosection vanishes. So we only have to compute the equivariant Euler characteristic of the virtual normal bundle
\[
N^{\vir}:=(\bE_{\fN^{\ccN}/\fB^{\tw}}^{\mov}|_{\fZ^{\tw}})^{\vee}=R^{\bullet}\pi_{\fZ^{\tw}*}(ev_{\fZ^{\tw}}^*\ccN\oplus ev_{\fZ^{\tw}}^*\ccN^{\vee}\otimes\omega_{\pi_{\fZ^{\tw}}}).
\]
Computing this at a point $[(C,\Sigma_1,\dots,\Sigma_n),f]$ of $\fZ^{\tw}(\lambda_1,\dots,\lambda_n)$ gives
\[
e_T(N^{\vir})=(-1)^{\chi(f^*N)}.
\]
Let $h=i\circ f:C\to \ccX$, then $\chi(f^*N)=\chi(h^*\ccE)=r(1-g)+d-\sum_{i=1}^n\age_{\ccX_{\lambda_i}}(\ccE)$.
 Then by \cref{cosloc} we have
\[
\iota_0^![\fV^{\tilde{\ccE}}(\lambda_1,\dots\lambda_n)]^{\vir}_{\sigma_{\fV^{\tilde{\ccE}}}}=[\fN^{\ccN}(\lambda_1,\dots,\lambda_n)/\fB^{\tw}]^{\vir}_{\sigma_{\fN^{\ccN}/\fB^{\tw}}}=(-1)^{r(1-g)+d-\sum_{i=1}^n\age_{\ccX_{\lambda_i}}(\ccE)}[\fZ^{\tw}(\lambda_1,\dots,\lambda_n)]^{\vir}.
\]
which together with \cref{definvDMcase} proves \cref{mainthmDM}.

\section{Examples}
As mentioned in the introduction, a big motivation for the introduction of the theory of stable maps with $p$-fields is the failure of the quantum Lefschetz hyperplane principle in higher genus. Interestingly, as observed in \cite{MR3039825}, quantum Lefschetz can fail even for positive orbifold hypersurfaces. The paper gives two explicit examples, which are interesting to revisit in light of the theory we have developed in this chapter. We consider a slightly more minimal ``trivial'' example than the authors, because it exhibits the same behavior. 

\subsection{Trivial example}
 Let $X=\bP[1,2,2]$ and $Z=\bP[2,2]$ the vanishing locus of the section $s=x_1$ of $E=\cO(1)$.
We have $IX=X_0\sqcup X_1$ where $X_0=X$ and $X_1=[\bP^1\times\{-1\}/\mu_2]\cong\bP[2,2]$. Similarly, $IZ=Z\sqcup X_1$.
Let $i:Z\xhookrightarrow{}X$, this coincides with the restriction of the morphism $IX\to X$ to $X_1$.
 Consider the moduli space of genus $0$, $4$-pointed degree $0$ twisted stable maps to $X$ with the condition that the evaluation maps at the marked points lands in $X_1$. In the notation of \cite{MR3039825}, this is $X_{0,\vec{4},0}=\fX^{\tw}(1,1,1,1)=:\fX$. Similarly, we have $ Z_{0,\vec{4},0}:=\fZ$ the moduli space of stable twisted maps to $Z$ of genus 0 and degree 0 with isotropy $\mu_2$ at the marked points. 
As DM stacks, these spaces are isomorphic, smooth, of pure dimension $2$:
\[
\fX\cong \fZ \cong \overline{\ccM}^{\tw}_{0,\vec{4}}\times\bP[2,2]=:M,
\]
where $\overline{\ccM}^{\tw}_{0,\vec{4}}$ is the moduli space of \textit{stable} twisted curves with 4 marked points each with isotropy $\mu_2$. This is because $\beta=0$ implies that the domain curve must itself be stable, so we just need to record the domain curve and a point in the chosen component of the inertia stack. 
The obstruction theories are however different, in fact we get different virtual dimensions $d^{\vir}_{\fX}=1$
whereas $d^{\vir}_{\fZ}=2$.
So it is clear from the get-go that it cannot be the case that $[\fZ]^{\vir}=e(\ccE)\cap[\fX]^{\vir}$ for a vector bundle $\ccE$ on $\fX$. For both moduli spaces, we have that their virtual fundamental class is $[M]\cap e(\Ob)$ with $\Ob$ the appropriate obstruction sheaf. Let the universal families of $\fX$ and $\fZ$ be denoted respectively by $(\ccC_{\fX},\pi,ev)$ and $(\ccC_{\fZ},\pi,ev)$. We have
\[
\Ob_{\fX}=R^1\pi_*(ev^*T_X)
\]
Consider a smooth point of $\overline{\ccM}_{0,\vec{4},0}$, which is an orbifold curve $(C,\Sigma_1,\dots,\Sigma_4)$ with coarse moduli space $|C|=\bP^1$ and four orbifold points $\Sigma_i\cong B\mu_2$. Such a curve, known as the pillowcase, arises by quotienting the a smooth genus 1 curve by $z\mapsto -z$. Note that the dualizing sheaf of $C$ is $\omega_C=\cO_C(\Sigma_1+\Sigma_2-\Sigma_3-\Sigma_4)=\omega_C$, which has degree 0. Let $\xi=[(C,\Sigma_1,\dots,\Sigma_4),f]$ be an object of $\fX$ over this, with for $f:C\to \bP[1,2,2]$ constant with image in $X_1$. 
\[
\Ob_{\fX}|_{\xi}=H^1(C,f^*T_X)
\]
The tangent bundle of $X$ sits in the Euler sequence:
\[
0\to \cO_X\to\cO_X(1)\oplus\cO_X(2)\oplus\cO_X(2)\to T_{X}\to 0
\]
from which we see $H^1(C,f^*T_X)=H^1(C,f^*(\cO_X(1)\oplus\cO_X(2)\oplus\cO_X(2)))$. $f$ has degree 0 and we get 
\[
f^*\cO_X(1)=\omega_C
\]
and
\[
f^*\cO_X(2)=\cO_C.
\]
Then $\Ob_{\fX}|_{\xi}=H^1(C,\omega_C)$. Globally, this is the line bundle $\ccL=R^1\pi_*\ev^*\cO_X(1)$. So the perfect obstruction theory on $\fX$ is
\[
\bE_{\fX}=[\ccL^{\vee}\xrightarrow{0} \Omega_M].
\]
On the other hand, the obstruction sheaf $\Ob_{\fZ}$ is trivial, as we can see by the same computation. 

Now consider the moduli of stable maps with fields $\fX^E$, which parametrizes $\overline{\xi}=[(C,\Sigma_1,\dots,\Sigma_n),f,p]$ with $f:C\to\bP[1,2,2]$ constant as before and 
\[
p\in H^0(C,f^*\cO_X(-1)\otimes\omega_C)=H^0(C,\omega_C^{\otimes 2})=H^0(C,\cO_C).
\]
Globally, $\fX^E$ is the total space of the dual of the obstruction bundle $\ccL=R^1\pi_*ev^*\cO_X(1)$. Let $q:\Vb(\ccL^{\vee})\to M$ be the projection. We choose a perfect obstruction theory on it given by
\[
\bE_{\fX^E}=[q^*\ccL^{\vee} \xrightarrow{0}\Omega_{\Vb(\ccL^{\vee})}]
\]
The section $s$ gives a map $\delta s: \cO_X(1)\oplus\cO_X(2)\oplus\cO_X(2)\to\cO_X(1)$ which is just projection to the first coordinate. We get a cosection
\begin{align*}
\sigma|_{\overline{\xi}}:\Ob_{\fX^E}|_{\overline{\xi}}=H^1(C,\omega_C) \to \cO_{C}\\
\dot{x_1}\mapsto \langle\dot{x_1}, p \rangle
\end{align*}
The degeneracy locus is the zero section $\{p=0\}=M\subset \fX^E$. The cosection gives an isomorphism $\Ob_{\fX^E}=q^*\ccL\to \cO_{\fX^E}(-M)$, the cosection localized virtual class is now just 
\[
[\fX^E]^{\vir}_{\sigma}=-[M]=-[\fZ]^{\vir}
\]
 as we expect from \cref{mainthmDM}.
 Incidentally, this is (the dual of) the situation described in \cite{MR3607000} with $N=\fX^E$. 

\subsection{Non-trivial example}
Let $X=\bP[1,1,1,2,2,2,2]$ and $Z=[1,1,2,2,2]$ be the vanishing locus of $s=x_1+x_7$ as a section of $E=\cO(1)\oplus\cO(2)$. Then $IX=X\sqcup \bP[2,2,2,2]$ and $IZ=Z\sqcup \bP[2,2,2]$. Again consider $\fX:=\fX^{\tw}(1,1,1,1)= \overline{\ccM}_{0,\vec{4},0}\times \bP[2,2,2,2]$ and $\fZ:=\fZ^{\tw}(1,1,1,1)=\overline{\ccM}_{0,\vec{4},0}\times \bP[2,2,2]\xhookrightarrow{\iota} \fX$ with universal families $(\ccC_\fX,\pi,ev)$ and $(\ccC_{\fZ},\pi,ev)$. These are smooth stacks of dimensions respectively 4 and 3. The obstruction sheaves are
\[
\Ob_{\fX}=R^1\pi_*ev^*\cO_X(1)^{\oplus 3}=\cL^{\oplus 3}
\]
and
\[
\Ob_{\fZ}=R^1\pi_*ev^*\cO_Z(1)^{\oplus 2}=\iota^*\cL^{\oplus 2}
\]
so both $\fX$ and $\fZ$ have virtual dimension 1. 
The moduli of stable maps with fields is still the line bundle $q: \fX^E=\Vb(\cL^{\vee})\to \fX$, because $\fX^E=\Spec\Sym(R^1\pi_*ev^*(\cO_X(1)\oplus\cO_X(2)))=\Spec_{\fX}\Sym(\cL)$. The relative perfect obstruction theory we consider is
\[
\bE_{\fX^E/\fX}^{\vee}=R^{\bullet}\pi_*(ev^*E^{\vee}\oplus\omega_{C_{\fX^E}})\cong [q^*\ccL^{\vee}\xrightarrow{0}q^*\cL]
\]
so we see that the absolute perfect obstruction theory
\[
\bE_{ \fX^E}^{\vee}=[T_{ \fX^E}\xrightarrow{0} q^*\cL^{\oplus 4}]
\]
gives the compatibility diagram
\[
\xymatrix{
\bT_{ \fX^E/ \fX}\ar[d]\ar[r] & \bT_{ \fX^E}\ar[d]\ar[r] &  q^*\bT_{ \fX}\ar[d]\ar[r]^{+1} &\\
\bE^{\vee}_{ \fX^E/ \fX}\ar[r] & \bE^{\vee}_{ \fX^E}\ar[r] &  q^*\bE^{\vee}_{ \fX}\ar[r]^{+1} & .
}
\]
The virtual dimension of $ \fX^E$ is 1. 
The local expression for the cosection at $\overline{\xi}=[(C,\Sigma_1,\dots,\Sigma_n),f,p]$ is
\begin{align*}
\Ob_{ \fX^E}|_{\overline{\xi}}: H^0(C,\omega_C)^{\oplus 4} \to \bC\\
(\dot{x_1},\dot{x_2},\dot{x_3},\dot{p})\mapsto \langle \dot{x_1}, p \rangle + \langle \dot{p},f^*s  \rangle
\end{align*}
where $f^*s\in H^0(C,f^*\cO_X(1)\oplus f^*\cO_X(2))=H^0(C,\cO_C)$ is just $f^*x_7$. The degeneracy locus is $\{p=f^*x_7=0\}$, which we can identify with $\overline{\ccM}_{0,\vec{4},0}\times\bP[2,2,2]\subset \overline{\ccM}_{0,\vec{4},0}\times\bP[2,2,2,2]\subset  \fX^E$.  
\cref{mainthmDM} gives is $[ \fX^E]_{\sigma}^{\vir}=(-1)^{\chi(f^*E)}[ \fZ]^{\vir}=[ \fZ]^{\vir}$.

\bibliography{references}
\bibliographystyle{amsplain} 
\end{document}